\documentclass[draft, 11pt]{article}

\usepackage{graphicx}
\usepackage{amssymb}
\usepackage{epstopdf}
\usepackage{amsmath,amsthm}
\usepackage{enumerate}
\usepackage{cite}
\usepackage{mathcomp}
\usepackage{supertabular}
\usepackage{longtable}
\usepackage{stmaryrd}
\usepackage{url}
\usepackage{multirow}
\usepackage{color}
\usepackage{rotating}
\usepackage{pdflscape}
\usepackage{float}
\usepackage{fullpage}

\begin{document}

\title{Generalized Balanced Tournament Packings and Optimal Equitable Symbol Weight Codes for Power Line Communications}

\author{Y.~M.~Chee \\
\small Division~of~Mathematical Sciences, School~of~Physical~and~Mathematical~Sciences, \\
\small Nanyang~Technological~University, 21~Nanyang~Link, Singapore~637371, Singapore,\\
\small email: {\tt ymchee@ntu.edu.sg} 
\and H.~M.~Kiah \\
\small Division~of~Mathematical Sciences, School~of~Physical~and~Mathematical~Sciences, \\
\small Nanyang~Technological~University, 21~Nanyang~Link, Singapore~637371, Singapore,\\
\small email: {\tt hmkiah@ntu.edu.sg} 
\and A.~C.~H.~Ling \\
\small Department of Computer Science, University of Vermont, Burlington, VT 05405, USA,\\
\small email: {\tt aling@emba.uvm.edu}
\and C.~Wang\\
\small School of Science, Jiangnan~University, Wuxi~214122, China,\\
\small email: {\tt wcm@jiangnan.edu.cn}
}

\theoremstyle{definition}
 
\newtheorem{theorem}{Theorem}[section]
\newtheorem{proposition}[theorem]{Proposition}
\newtheorem{definition}[theorem]{Definition}
\newtheorem{lemma}[theorem]{Lemma}
\newtheorem{corollary}[theorem]{Corollary}

\theoremstyle{remark}
\newtheorem{claim}{Claim}[section]
\newtheorem{example}{Example}[section]


\newcommand{\vA}{{\sf A}}
\newcommand{\vB}{{\sf B}}
\newcommand{\vC}{{\sf C}}
\newcommand{\vD}{{\sf D}}
\newcommand{\vE}{{\sf E}}
\newcommand{\vF}{{\sf F}}

\newcommand{\vW}{{\sf W}}

\newcommand{\vc}{{\sf c}}
\newcommand{\vu}{{\sf u}}
\newcommand{\vv}{{\sf v}}
\newcommand{\vw}{{\sf w}}
\newcommand{\tA}{\textrm A}
\newcommand{\tB}{\textrm B}
\newcommand{\A}{\mathcal A}
\newcommand{\B}{\mathcal B}
\newcommand{\C}{\mathcal C}
\newcommand{\D}{\mathcal D}
\newcommand{\G}{\mathcal G}
\newcommand{\R}{\mathcal R}
\newcommand{\SSS}{\mathcal S}

\newcommand{\enarrow}{e_{\sf NBD}}
\newcommand{\efade}{e_{\sf SFD}}
\newcommand{\eimpulse}{e_{\sf IMP}}
\newcommand{\einsert}{e_{\sf INS}}
\newcommand{\edelete}{e_{\sf DEL}}
\newcommand{\ec}{E_{\mathcal C}}

\newcommand{\union}{\bigcup\limits}
\newcommand{\swt}{\textrm {swt}}

\newcommand{\CC}{\mathbb C} 
\newcommand{\RR}{\mathbb R}
\newcommand{\ZZ}{\mathbb Z}
\newcommand{\FF}{\mathbb F}
\newcommand{\ceiling}[1]{\left\lceil{#1}\right\rceil}
\newcommand{\floor}[1]{\left\lfloor{#1}\right\rfloor}

\newcommand{\wt}[1]{\textrm{wt}{(#1)}}
\newcommand{\trace}[1]{\textrm{Trace}{(#1)}}
\newcommand{\lev}[1]{\textsf{lev}{(#1)}}
\newcommand{\dist}{\textsf{dist}}
\newcommand{\packing}[1]{\textrm{Packing}_\textrm{#1}}
\newcommand{\ppty}[1]{\textsf{Property #1}}
\newcommand{\pp}{^\prime}

\newcommand{\block}{\mathcal B}
\newcommand{\gbtd}{\textrm{GBTD}}
\newcommand{\gbtp}{\textrm{GBTP}}
\newcommand{\bibd}{\textrm{BIBD}}
\newcommand{\rbibd}{\textrm{RBIBD}}
\newcommand{\pbd}{\textrm{PBD}}
\newcommand{\drtd}{\textrm{DRTD}}
\newcommand{\kts}{\textrm{KTS}}
\newcommand{\fkts}{\textrm{FKTS}}
\newcommand{\td}{\textrm{TD}}
\newcommand{\ttd}{\textrm{TTD}}

\newcommand{\inprod}[1]{\langle{#1}\rangle}


\newcommand{\beas}{\begin{eqnarray*}}
\newcommand{\eeas}{\end{eqnarray*}}

\newcommand{\bm}[1]{{\mbox{\boldmath $#1$}}} 

\newcommand{\tworow}[2]{\genfrac{}{}{0pt}{}{#1}{#2}}
\newcommand{\qbinom}[2]{\left[ {#1}\atop{#2}\right]_q}

\newcommand{\Lovasz}{Lov\'{a}sz }
\newcommand{\citereq}{[citation required]}
\newcommand{\etal}{\emph{et al.}}

\maketitle

\renewcommand{\thefootnote}{\fnsymbol{footnote}}

 \begin{abstract}
 Generalized balance tournament packings (GBTPs) extend the
 concept of generalized balanced tournament designs 
 introduced by Lamken and Vanstone (1989).
 In this paper, we establish the connection between GBTPs
 and a class of codes called equitable symbol weight codes.
 The latter were recently demonstrated to optimize
 the performance against narrowband noise in a general coded
 modulation scheme for power line communications. 
 By constructing classes of GBTPs, we establish infinite families of
  optimal equitable symbol weight codes
 with code lengths greater than alphabet size and
 whose narrowband noise error-correcting capability to code length ratios
 do not diminish to zero as the length grows.

 \end{abstract}

 
 \pagebreak

\section{Introduction}
Power line communications (PLC) is a 
technology that enables the transmission of data over 
electric power lines.
It was started in the 1910's for voice communication \cite{Schwartz:2009}, and used in
the 1950's in the form of ripple control for load and tariff management
in power distribution.
With the emergence of the Internet in the 1990's, research into broadband PLC
gathered pace as a promising technology for Internet access and local area networking, since
the electrical grid infrastructure provides ``last mile'' connectivity
to premises and capillarity within premises. Recently, there has been a renewed interest
in high-speed narrowband PLC due to applications in sustainable energy strategies, specifically in
smart grids (see \cite{Haidineetal:2011,Dzungetal:2011,Liuetal:2011,ZhangYang:2011}).
However, power lines present a difficult communications environment and
overcoming permanent narrowband disturbance has remained a
challenging problem \cite{Vinck:2000,Biglieri:2003,Pavlidouetal:2003}.
Vinck \cite{Vinck:2000} addressed this problem through the use of a 
coded modulation scheme based on permutation codes.
More recently,
Chee \etal \cite{Cheeetal:2013tc} extended Vinck's analysis to general block codes
and motivated the study of {\em equitable symbol weight codes} (ESWCs).

Relatively little is known about optimal ESWCs, 
other than those that correspond to
permutation codes, injection codes and frequency permutation arrays.
In particular, only six infinite families
of optimal ESWCs 
with code length greater than alphabet size are known.
These have all been constructed by
Ding and Yin \cite{DingYin:2006}, and Huczynska and Mullen \cite{Huczynska:2010}
as frequency permutation arrays and they meet
the Plotkin bound.
One drawback with the code parameters of these families
is that the narrowband noise error-correcting capability to length ratio 
diminishes as length grows.

In this paper, we construct infinite families of optimal ESWCs
 whose code lengths are larger than alphabet size and
whose narrowband noise 
error-correcting capability to length ratios tend to a positive constant as code length grows. These
families of codes all attain the generalized Plotkin bound. 
Our results are based on the construction of equivalent combinatorial designs called
generalized balanced tournament packings (GBTPs).

GBTPs extend the concept of generalized balanced tournament designs (GBTDs)
introduced by Lamken and Vanstone \cite{LamkenVanstone:1989}.
GBTDs have been extensively studied \cite{Lamken:1990,Lamken:1992,Lamken:1994,Yinetal:2008,Cheeetal:2013b,Daietal:2013}
and is useful in the constructions of  resolvable, near-resolvable, doubly resolvable, and 
doubly near-resolvable balanced incomplete block designs \cite{Lamken:1990,Lamken:1995,Lamken:1997}.
Using the classical correspondence given by Semakov and Zinoviev\cite{SemakovZinoviev:1968}
(see also \cite{DingYin:2005b, Bogdanovaetal:2007,Yinetal:2008}),
we construct optimal families of ESWCs from certain families of GBTPs.
We establish existence results for these families of GBTPs
by borrowing standard recursion and 
direct construction methods from combinatorial design theory.

The paper is organized as follows. 
In Section 2, we introduce ESWCs and
survey the known results on optimal codes. 
In Section 3, we introduce GBTPs and establish the equivalence 
between GBTPs and ESWCs. 
At the end of the section, we establish two classes of GBTPs that 
correspond to optimal equitable symbol weight codes. 
In Sections 4 to 7, we settle the existence of these two classes of GBTPs.
Section 4 outlines the general strategy, while 
Section 5 and Section 6 provide recursive and direct constructions respectively.

Some of the results of the paper have been initially reported 
at IEEE International Symposium on Information Theory 2012 \cite{Cheeetal:2012b},
and the present paper contains detailed proofs and 
includes a new existence result on a family of GBTPs with block size two and three.

 \section{Preliminaries}

 \subsection{Notations}

For positive integer $m$ and prime power $q$,
 denote the ring $\ZZ/m\ZZ$  by $\ZZ_m$ 
 and the finite field of $q$ elements by $\FF_q$.
 Let $\ZZ_{>0}$ denote the set of positive integers.
Let $[m]$ denote the set $\{1,2,\ldots, m\}$. 
We use angled brackets ($\langle$ and  $\rangle$) for multisets. Disjoint set
union is depicted using $\sqcup$.
For sets $A$ and $B$, an element $(a,b)\in A\times B$ is sometimes
written as $a_b$ for succinctness.


A {\em set system} is a pair ${\frak S}=(X,\A)$, where $X$ is a
finite set of {\em points} and $\A\subseteq 2^X$. Elements of $\A$
are called {\em blocks}. The {\em order} of $\frak S$ is the number
of points in $X$, and the {\em size} of $\frak S$ is the number of
blocks in $\A$. Let $K$ be a set of nonnegative integers. The set
system $(X,\A)$ is said to be {\em $K$-uniform} if $|A|\in K$ for
all $A\in\A$.

 \subsection{Equitable Symbol Weight Codes}
\label{prelim}


Let $\Sigma$ be a set of $q$ {\em symbols}. A {\em $q$-ary} {\em
code} of {\em length} $n$ over the {\em alphabet} $\Sigma$ is a
subset $\C\subseteq\Sigma^n$. Elements of $\C$ are called {\em
codewords}. The {\em size} of $\C$ is the number of codewords in
$\C$. For $i\in[n]$, the $i$th coordinate of a codeword $\vu\in\C$
is denoted $\vu_i$, so that $\vu=(\vu_1,\vu_2,\ldots,\vu_n)$.
Denote the {\em frequency} of symbol $\sigma\in \Sigma$ in codeword
$\vu\in\Sigma^n$ by $w_\sigma(\vu)$, that is, $w_\sigma(\vu) = |
\{\vu_i=\sigma : i\in[n]\}|$.

An element $\vu\in\Sigma^n$ is said to have \emph{equitable symbol
weight}  if $w_\sigma(\vu)\in \{\lfloor n/q\rfloor,\lceil
n/q\rceil\}$ for any $\sigma\in \Sigma$. If all the codewords of
$\C$ have equitable symbol weight, then the code $\C$ is called an
\emph{equitable symbol weight code} (ESWC).
Consider the usual Hamming distance defined on codewords and codes
and let $d$ denote the minimum distance of a code $\C$.
In addition, consider the following parameter.

\begin{definition}
Let $\mathcal C$ be a $q$-ary code with minimum distance $d$. The
{\em narrowband noise error-correcting capability}
of $\mathcal C$ is
\begin{equation*}
c({\mathcal C})=\min\{e: E_\C(e)\ge d\},
\end{equation*}
where $E_\C$ is a function $E_\C:[q]\to[n]$, given by
\begin{equation*}
E_{\mathcal C}(e)=\max_{\Gamma\subseteq\Sigma \atop |\Gamma|=e}
\max_{\vc \in\mathcal C} \left\{\sum_{\sigma \in
\Gamma}w_\sigma(\vc)\right\}.
\end{equation*}
\end{definition}

Chee \etal{} \cite{Cheeetal:2013tc} established that a code
$\mathcal C$ can correct up to $c({\mathcal C})-1$ narrowband noise
errors and demonstrated that an ESWC
maximizes the quantity $c(\C)$, for fixed $n$, $d$ and $q$.

Henceforth, only ESWCs are considered. A
$q$-ary ESWC of length $n$ having minimum
distance $d$ is denoted ESWC$(n,d)_q$. 
Denote the maximum size of an ESWC$(n,d)_q$ by $A^{ESW}_q(n,d)$. 
Any ESWC$(n,d)_q$ of size $A^{ESW}_q(n,d)$ is said to be {\em optimal}.
Taken as a $q$-ary code of length $n$, an optimal
ESWC$(n,d)_q$ satisfies the generalised Plotkin bound
\cite{Bogdanovaetal:2001}.

\begin{theorem}[Generalised Plotkin Bound]\label{plotkin}
If there is a ESWC$(n,d)_q$ $\C$ of size $M$, then
\begin{equation}\label{eq:plotkin}
\binom{M}{2} d \le n\sum_{i=0}^{q-2} \sum_{j=i+1}^{q-1} M_iM_j,
\end{equation}
\noindent where $M_i=\floor{(M+i)/q}$.
If $q$ divides $M$ and
$\binom{M}{2} d = n\binom{q}{2}(M/q)^2$,
then $\C$ is optimal.
\end{theorem}


In the rest of this paper, ESWCs whose sizes attain the generalised Plotkin bound are constructed.
In particular, the following is established.

\begin{theorem}\label{thm:main} The following holds.
 \begin{enumerate}[(i)]
 \item \begin{equation*}
  A^{ESW}_q(2q-1,2q-2)=
 \begin{cases}
 3,         & q=2,\\
 2q,    &q\geq 3.
 \end{cases}
 \end{equation*}

 \item \begin{equation*}
 A^{ESW}_q(3q-1,3q-3)=
 \begin{cases}
 4,    & q=2, \\
 3q,  &q\geq 3.
 \end{cases}
 \end{equation*}

 \item \begin{equation*}
 A^{ESW}_q(4q-1,4q-4)=
 \begin{cases}
 4q-1,        & q=2, 3,\\
 4q,  &q\geq 4.

 \end{cases}
 \end{equation*}

 \item If $q\ge 62$ or $q\in \{5-18,30,42,46,48-50,54-57\}$,
  \begin{equation*}
 A^{ESW}_q(5q-1,5q-5)=
 5q. 
 \end{equation*}
 
 \item If $q$ is an odd prime power,
 \begin{equation*}
 A^{ESW}_q(q^2-1,q^2-q)=
 q^2. 
 \end{equation*}

 \item \begin{equation*}
 A^{ESW}_q\left(\frac{3q-1}2,\frac{3q-3}{2}\right)=
 \begin{cases}
 4q-6,        & q=3, 5, \\
 3q,  &q\geq 7 \mbox{ is odd}.\\
 \end{cases}
 \end{equation*}
 
  \item \begin{equation*}
 A^{ESW}_q(2q-3,2q-4)=
 \begin{cases}
 6q-12,      & q=3,4, \\
 14,        &q=5,6,\\
 2q+1,      &q\geq 7, \mbox{except possibly $q\in\{12,13\}$.}
 \end{cases}
 \end{equation*}
 
 \end{enumerate}

\end{theorem}

Observe that any ESWC $\C$ with the above parameters
must have $c(\C)=q-1$.
In Table \ref{table:optimal}, we verify that $c(\C)/n$ tends to a positive constant as $q$ grows.
In the same table, we compare with known families of optimal
ESWC$(n,d)_q$.

In particular,
only six infinite nontrivial families of optimal codes with $n> q$ are known. However,
code parameters for these six families are such that their
relative narrowband noise 
error-correcting capability to length ratios diminish to zero as $q$ grows.
This is undesirable for narrowband noise correction for PLC.
Hence, Theorem \ref{thm:main} provides infinite families of optimal ESWCs
with code lengths are larger than alphabet size and
whose relative narrowband noise capability to length ratios tend to a positive constant as length grows.

These optimal ESWCs are constructed from GBTPs
using the classical correspondence given by Semakov and Zinoviev\cite{SemakovZinoviev:1968}%
\footnote{Bogdanova \etal{} \cite{Bogdanovaetal:2007} gave a survey of connection between equidistant codes and designs.
Using this correspondence, Ding and Yin \cite{DingYin:2005b} constructed optimal constant-composition codes,
while Yin \etal{} \cite{Yinetal:2008} constructed near-constant-composition codes.}.
We remark that GBTPs extend the concept of GBTDs 
and consequently Theorem \ref{thm:main} (i) to (v) follows directly from known classes of GBTDs.
We explain the connection in detail in the next section.

\begin{landscape}
\begin{table}
\renewcommand{\arraystretch}{1.1}
\centering
\small
\caption{Infinite families of optimal ESWC$(n,d)_q$}
\label{table:optimal}
\begin{tabular}{p{8cm}cccp{37mm}}\hline
ESWC$(n,d)_q$ $\C$ & $|\C|$ & $c(\C)$ & $\underset{q\to \infty}{\lim} c(\C)/n$ & Remarks\\ \hline\hline
$(n,n)_q$ for $q\ge 2$ & $q$ & $\min\{n,q\}$ & --
& easy \\ \hline
$(3,2)_q$  for $q\geq 3$  & $q(q-1)$ & 2 & 0 & injection code \cite{Dukes:2012}\\

$(4,2)_q$ for $q\geq 4$, $q\not=7$ & $q(q-1)(q-2)$ & 2 & 0 & injection code \cite{Dukes:2012} \\

$(n,1)_q$ for $n< q$     & $q(q-1)\cdots(q-n+1)$ & 1 & $1/n$ & injection code, easy\\
$(qn,2)_q$ for $q\geq 2$     & $(qn)!/(n!)^q$  & 2 & 0 &  frequency permutation array, easy\\
$(q,3)_q$ for $q\geq 3$ & $q!/2$ & 3 & 0 & permutation code, easy\\
$(q,q-1)_q$ for prime powers $q$        & $q(q-1)$ & $q-1$ & 1 & injection code \cite{Colbournetal:2004}\\
$(n,n-1)_q$ for $q$ sufficiently large and $n\le q$     & $q(q-1)$ & $n-1$ & $1-1/n$ & injection code \cite{Dukes:2012}\\
$(q,q-2)_q$ for prime powers $q-1$      & $q(q-1)(q-2)$ & $q-2$ & 1&  permutation code \cite{FranklDeza:1977}\\ \hline

$\left(q(q+1),q^2\right)_q$ for  prime powers $q$ & $q^2$ & $q$ & 0 & frequency permutation array \cite{DingYin:2005b} \\

$\left(\frac{q(kq^2-1)}{k-1},\frac{kq^2(q-1)}{k-1}\right)_q$ for prime powers $q$, $2\leq k\leq 5$, $(k,q)\not=(5,9)$
& $kq^2$  & $q$ & 0 & frequency permutation array\cite{DingYin:2006}\\
$\left(\frac{\mu q^{s-t}(q^{2s-t}-1)}{q^t-1}, \frac{\mu q^{2s-t}(q^{s-t}-1)}{q^t-1}\right)_{q^{s-t}}$
for prime powers $q$,  $1\le t < s$,
$\mu=\prod_{i=1}^{t-1}\frac{q^{s-i}-1}{q^i-1}$
& $q^{2s-t}$ &  $q^{s-t}$ & 0 & frequency permutation array\cite{DingYin:2006}\\

$(q^s(q^{2s+c}-1), q^{2s+c}(q^s-1))_{q^s}$,  for prime powers
$q$, and  $s,c\ge 1$ &
$q^{2s+c}$ & $q^s$ & $0$ & frequency permutation array\cite{DingYin:2006}\\

$({kq\choose k},\frac{kq-k}{kq-1}{kq\choose k})_q$ for  $q,k\ge 1$ &
$kq$ & $q-1$ & 0 &frequency permutation array\cite{Huczynska:2010}\\

$(2q^2-q,2q^2-2q)_q$ for even $q$, $q\notin\{2,6\}$ &
$2q$ & $q$ & 0 & frequency permutation array\cite{Huczynska:2010}\\ \hline

$(2q-1,2q-2)_q$ for $q\ge 3$ &
$2q$ & $q-1$ & $1/2$ & Theorem \ref{thm:main}\\

$(3q-1,3q-3)_q$ for $q\ge 3$ &
$3q$ & $q-1$ & $1/3$ & Theorem \ref{thm:main}\\

$(4q-1,4q-4)_q$ for $q\ge 4$ &
$4q$ & $q-1$ & $1/4$ & Theorem \ref{thm:main}\\

$(5q-1,5q-5)_q$ for $q\ge 62$ &
$5q$ & $q-1$ & $1/5$ & Theorem \ref{thm:main}\\

$(q^2-1,q^2-q)_q$ for $q\ge 4$ &
$q^2$ & $q-1$ & $0$ & Theorem \ref{thm:main}\\

$\left(\frac {3q-1}{2},\frac{3q-3}{2}\right)_q$ for $q\ge 7$ and $q$ odd &
$3q$ & $q-1$ & $2/3$ & Theorem \ref{thm:main}\\

$(2q-3,2q-4)_q$ for $q\ge 14$ &
$2q+1$ & $q-2$ & $1/2$ & Theorem \ref{thm:main}\\

\hline
\end{tabular}
\end{table}
\end{landscape}

 \section{Constructions of Equitable Symbol Weight Codes}

We first determine  $A_q^{ESW}(n,d)$ for small values of $n$, $q$ and $d$.
With the exception of $A_6^{ESW}(9,8)$, an exhaustive computer search
established the following values of $A_q^{ESW}(n,d)$. 
For $A_6^{ESW}(9,8)$,
an ESWC$(9,8)_6$ of size 14 was found via computer search.
Since an ESWC$(9,8)_6$ of size 15 cannot exist by
the generalized Plotkin bound, it follows that $A_6^{ESW}(9,8)=14$.
We record the results of the computations in the following proposition
and the corresponding optimal codes can be found at
\cite{Cheeetal:2012online}.

\begin{proposition}\label{prop:small}
The following holds:
\begin{align*}
A_2^{ESW}(3,2) &=3     &     A_2^{ESW}(5,3) &=4     &     A_2^{ESW}(7,4) &=7\\
A_3^{ESW}(3,2) &=6     &     A_3^{ESW}(4,3) &=6     &     A_3^{ESW}(11,8) &=11\\
A_4^{ESW}(5,4) &=12   &     A_5^{ESW}(7,6) &=14   &     A_6^{ESW}(9,8) &=14.
\end{align*}
\end{proposition}
The rest of the paper establishes the remaining values 
in Theorem \ref{thm:main}. 
To do so, we define a class of combinatorial designs that is equivalent to ESWCs.

\subsection{Equitable Symbol Weight Codes and Generalized Balanced Tournament Packings}
 Let $\lambda$, $v$ be positive integers and $K$ be a set of nonnegative integers.
 A {\em $(v,K,\lambda)$-packing} is a $K$-uniform set
 system of order $v$ such that every pair of distinct points is contained in at most $\lambda$ blocks. 
 The value $\lambda$ is called the {\em index} of the packing.
 A {\em parallel class} (or {\em resolution class}) of a packing is a subset of the blocks that partitions the
 set of points $X$. If the set of blocks can be partitioned into
 parallel classes, then the packing is {\em resolvable}, and denoted
 by RP$(v,K,\lambda)$. An RP$(v,K,\lambda)$ is called a {\em maximum  resolvable packing}, denoted
 by MRP$(v,K,\lambda)$, if it contains maximum possible number of
 parallel classes.

 Furthermore, an MRP$(v,\{k\},\lambda)$ is called a {\em resolvable $(v,\{k\},\lambda)$-balanced incomplete block design}, or ${\rm RBIBD}(v,k,\lambda)$ in short, if every pair of distinct points
 is contained in exactly $\lambda$ blocks. A simple computation gives
  the size of an ${\rm RBIBD}(v,k,\lambda)$ to be
 $\frac{\lambda v(v-1)}{k(k-1)}$.
 
 We define the combinatorial object of study in this paper. 
 We note that this definition is a generalization of GBTDs to packings
 and various indices.
                                                        
 \begin{definition}
 Let $(X,\A)$ be an RP$(v,K,\lambda)$ with $n$ parallel classes.
 Then $(X,\A)$ is called a {\em generalized balanced tournament packing}
 if the blocks of $\A$ are arranged into an $m\times n$ array satisfying the
 following conditions:
 \begin{enumerate}[(i)]
 \item every point in $X$ is contained in exactly one cell of each column,
 \item every point in $X$ is contained in either  $\ceiling{n/m}$ or $\floor{n/m}$ cells of each row.
 \end{enumerate}
 We denote such a GBTP by $\gbtp_\lambda(K;v,m\times n)$.
 \end{definition}

 Unless otherwise stated, the rows of a $\gbtp_\lambda(K;v,m\times n)$ are
 indexed by $[m]$ and the columns by $\left[ n\right]$.

 In a ${\rm GBTP}_\lambda(K;v,m\times n)$, given point $x$ and
 column $j$, there is a unique row that contains the point $x$ in
 column $j$. Hence, for each point $x\in X$ of a ${\rm
 GBTP}_\lambda(K;v,m\times n)$ $(X,\A)$, we may correspond the
 codeword $\vc(x)=\left(r_1,r_2,\ldots,r_n\right)\in [m]^n$, where
 $r_j$ is the row in which point $x$ appears in column $j$. It is
 obvious that $\C=\{\vc(x): x\in X\}$ is an $m$-ary code of length
 $n$ over the alphabet $[m]$. We note that this correspondence is
 precisely the one used by Semakov and Zinoviev
 \cite{SemakovZinoviev:1968} to show the equivalence between {\em
 equidistant codes} and resolvable balanced incomplete block designs.

 For distinct points $x,y\in X$, the distance between $\vc(x)$ and
 $\vc(y)$ is the number of columns for which $x$ and $y$ are not both
 contained in the same row. Since there are at most $\lambda$ blocks
 containing both $x$ and $y$, and that no two such blocks can occur
 in the same column of the ${\rm GBTP}_\lambda(K;v,m\times n)$, the
 distance between $\vc(x)$ and $\vc(y)$ is at least $n- \lambda$.

Next, we determine $w_i(\vc(x))$, for $x\in X$ and $i\in [m]$. From
the construction of $\vc(x)$, the number of times a symbol $i$
appears in $\vc(x)$ is the number of cells in row $i$ that contains
$x$. By the definition of a ${\rm GBTP}_\lambda(K;v,m\times n)$,
this number belongs to $\left\{\floor{n/m},\ceiling{n/m}\right\}$.
Hence, $\C$ is an ESWC of size $v$.
Finally, this construction of an ESWC from a
GBTP can easily be reversed. 
We record these observations as:

 \begin{theorem}
 \label{thm:esw-gbtp} Let $K$ be set of nonnegative integers.
  Then a ${\rm GBTP}_\lambda(K;v,m\times
 n)$ exists if and only if an ESWC$(n,n-\lambda)_m$ 
 of size $v$ exists.
 \end{theorem}
 
 We note that the correspondence between GBTPs and ESWCs was observed by Yin \etal{} \cite[Theorem 2.2]{Yinetal:2008}.
 However, in the latter paper, the class of codes constructed is called near-constant-composition codes (NCCCs). 
 Indeed, an NCCC is a special class of ESWC and one observes that  an ESWC$(n,d)_q$ is an NCCC when $n+1\equiv 0\bmod q$.

 \begin{example}
\label{example2,3} Consider the $\gbtp_1(\{2,3\},6,3\times 4)$
below.
\begin{center}
\begin{tabular}{|c|c|c|c|c|} \hline
             & $\{1,4\}$& $\{2,6\}$ & $\{3,5\}$\\ \hline
$\{1,2,3\}$  & $\{2,5\}$& $\{3,4\}$ & $\{1,6\}$\\ \hline
$\{4,5,6\}$  & $\{3,6\}$& $\{1,5\}$ & $\{2,4\}$\\ \hline
\end{tabular}
\end{center}
Each point $x\in[6]$ gives a codeword $\vc(x)=(r_1,r_2,\ldots,r_5)$,
where $r_j$ is the row in which point $x$ appears in column $j$.
Hence, we have
\begin{align*}
\vc(1) &= (2,1,3,2), &
\vc(2) &= (2,2,1,3), &
\vc(3) &= (2,3,2,1), \\
\vc(4) &= (3,1,2,3), &
\vc(5) &= (3,2,3,1), &
\vc(6) &= (3,3,1,2).
\end{align*}
The code $\C=\{\vc(1),\vc(2),\vc(3),\vc(4),\vc(5),\vc(6)\}$ is an
ESWC$(4,3)_3$ 
of size six.
\end{example}

 Theorem \ref{thm:esw-gbtp} sets up the equivalence between GBTPs
 and ESWCs.
In general, a GBTP may not correspond to an optimal ESWC.
However, in the following, we look at specific $K$'s to derive
families of optimal ESWCs.

 \subsection{Optimal Equitable Symbol Weight Codes from Generalized Balanced Tournament Designs}

 A ${\rm GBTP}_\lambda\left (\{k\};km,m\times \frac{\lambda(km-1)}{k-1}\right)$ is called a
 {\em generalized balanced tournament design} (GBTD), denoted by $\gbtd_\lambda(k,m)$.
 In this case, we check that each pair of distinct points is contained in exactly $\lambda$
 blocks and every point is contained in either
 $\ceiling{\frac{\lambda(km-1)}{m(k-1)}}$ or $\floor{\frac{\lambda(km-1)}{m(k-1)}}$ cells of each row.

Applying Theorem \ref{thm:esw-gbtp}, a ESWC$\left(\frac{\lambda(km-1)}{k-1},\frac{\lambda
 k(m-1)}{k-1}\right)_m$ of size $km$ exists and the corresponding code is optimal
 by generalized Plotkin bound. So, we have the following.


 \begin{theorem}
 \label{thm:esw-gbtd} A $\gbtd_\lambda(k,m)$ exists 
 if and only if an optimal ESWC$\left(\frac{\lambda(km-1)}{k-1},\frac{\lambda
 k(m-1)}{k-1}\right)_m$ of size
 $km$ exists and attains the generalized Plotkin bound.
%
 \end{theorem}


We remark that our definition of a GBTD extends that of
Lamken and Vanstone\cite{LamkenVanstone:1989}, which corresponds in our definition to
the case when $\lambda=k-1$.
The following summarizes the state-of-the-art results on the existence of
$\gbtd_{k-1}(k,m)$.

\begin{theorem}[Lamken \cite{LamkenVanstone:1989,Lamken:1990,Lamken:1992,Lamken:1994}, Yin \etal \cite{Yinetal:2008}, Chee \etal\cite{Cheeetal:2013b}, Dai \etal \cite{Daietal:2013}]
\label{lamkenyin}
The following holds.
\begin{enumerate}[(i)]
\item A $\gbtd_1(2,m)$ exists if and only if $m=1$ or $m\ge 3$.
\item A $\gbtd_2(3,m)$ exists if and only if $m=1$ or $m\ge 3$.
\item A $\gbtd_3(4,m)$ exists if and only if $m=1$ or $m\ge 4$.
\item A GBTD$_4(5,m)$ exists if $m\ge 62$ or $m\in \{5-18,30,42,46,48-50,54-57\}$.
\item A GBTD$_{k-1}(k,k)$ exists if $k$ is an odd prime power.
\end{enumerate}
\end{theorem}

Theorem \ref{thm:main} (i) to (v)
 is now an immediate consequence of Theorem \ref{thm:esw-gbtd}, \ref{lamkenyin}
and Proposition \ref{prop:small}.
The existence of $\gbtd_\lambda(k,m)$ when $\lambda\not=k-1$ has not been previously
investigated. The smallest open case is when $k=3$ and $\lambda=1$, which is the
case dealt with in this paper.

It follows readily from the fact that a $\gbtd_1(3,m)$ is also an ${\rm RBIBD}(3m,3,1)$,
that a necessary condition for a $\gbtd_1(3,m)$ to exist is that $m$ must be odd.
We note from Proposition \ref{prop:small} that $A^{ESW}_3(4,3)=6$ and $A^{ESW}_5(7,6)=14$, which do not meet the Plotkin bound.
Hence, the corresponding designs $\gbtd_1(3,3)$ and $\gbtd_1(3,5)$ do not exist by Theorem \ref{thm:esw-gbtd}.


Hence, a $\gbtd_1(3,m)$ can exist only if $m$ is odd and $m\notin \{3,5\}$. In Sections 4 to 7, we prove that this
necessary condition is also sufficient for the existence of $\gbtd_1(3,m)$. A direct consequence
of this is Theorem \ref{thm:main} (vi).

 \subsection{Optimal Equitable Symbol Weight Codes from GBTP$_1(\{2,3^*\};2m+1,m\times (2m-3))$}

 Theorem \ref{thm:esw-gbtd} constructs optimal ESWCs from GBTDs.
 In this subsection, we make slight variations to
 obtain another infinite family of optimal ESWCs. 

Consider a GBTP$_1(\{2,3\};v,m\times n)$. If there is
 exactly one block of size three in each resolution class,
 then we denote the GBTP by GBTP$_1(\{2,3^*\};v,m\times n)$.
 A simple computation then shows $v=2m+1$. 
 Now we establish the following construction for optimal ESWCs.

 \begin{theorem}
 \label{thm:optimalesw-gbtp} Let $m\ge 7$. If there exists a $\gbtp_1(\{2,3^*\};2m+1,m\times (2m-3))$,
 then there exists an optimal ESWC$(2m-3,2m-4)_m$
  of size $2m+1$ which attains the generalized Plotkin bound.
 \end{theorem}
 \begin{proof}
 By Theorem \ref{thm:esw-gbtp}, we have a ESWC$(2m-3,2m-4)_m$ 
 of size $2m+1$. It remains to verify its optimality.

Suppose otherwise that there exists an ESWC$(2m-3,2m-4)_m$ of size
 $2m+2$. Consider (\ref{eq:plotkin}) in Theorem \ref{plotkin}. On the left hand side, we have
 \begin{equation*}
\binom{2m+2}{2} \cdot (2m-4) =4m^3 - 2m^2 - 10m - 4.
 \end{equation*}
 Since $\floor{\frac{2m+2+i}{m}}=2$ for $0\le i\le m-3$ and $\floor{\frac{2m+2+(m-2)}{m}}=\floor{\frac{2m+2+(m-1)}{m}}=3$,
 the term on the right hand is
  \begin{align*}
  &(2m-3)\left(\left(\sum_{i=0}^{m-3}  4(m-3-i)+12\right)+9\right)\\
  &=(2m-3)(4m(m-2)-2(m-3)(m-2)+9)\\
  &=4m^3 - 2m^2 - 12m + 9
  \end{align*}

  But for $m\ge 7$,
  \begin{equation*}
4m^3 - 2m^2 - 10m - 4 > 4m^3 - 2m^2 - 12m + 9,
 \end{equation*}
\noindent contradicting (\ref{eq:plotkin}). Hence, an ESWC$(2m-3,2m-4)_m$
of size  $2m+2$ does not exist and the result follows.
 \end{proof}

In the rest of this paper, 
we construct a GBTP$_1(\{2,3^*\};2m+1,m\times (2m-3))$ for $m\ge 4$, except possibly $m\in\{12,13\}$.
This with Theorem \ref{thm:optimalesw-gbtp} and Proposition \ref{prop:small} gives Theorem \ref{thm:main} (vii).

\section{Proof Strategy of Theorem \ref{thm:main} (vi) and Theorem \ref{thm:main} (vii)}
 For the rest of the paper, we 
 determine with finite possible exceptions the existence of
 GBTD$_1(3,m)$ and GBTP$_1(\{2,3^*\};2m+1,m\times (2m-3))$.
 Our proof is technical and rather complex.
 However, it follows the general strategy of the previous work \cite{Lamken:1992,Yinetal:2008,Daietal:2013}.
 This section outlines the general strategy used, and introduces some required
 combinatorial designs.

 As with most combinatorial designs, direct construction to settle their existence
 is often difficult. Instead, we develop a set of recursive constructions, building
 big designs from smaller ones. Direct methods are used to construct a large enough
 set of small designs on which the recursions can work to generate all larger designs.
 For our recursion techniques to work, the GBTPs 
 must possess more structure than stipulated in its definition.
 First, we consider $\gbtd_1(3,m)$s that are
 {\em $*$colorable} which are defined below.

 \subsection{$c$-$*$colorable Generalized Balanced Tournament Designs}

We generalize the notion of factored GBTDs (FGBTDs) introduced by Lamken \cite{Lamken:1994}.
FGBTDs are crucial in the $k$-tupling construction for GBTDs of index $k-1$.
However, when the index is one, we extend this notion to $*$-colorability.

 \begin{definition}
 Let $c$ be positive. A {\em $c$-$*$colorable} ${\rm
 RBIBD}(v,k,\lambda)$ is an ${\rm RBIBD}(v,k,\lambda)$ with the property that its
 $\frac{\lambda v(v-1)}{k(k-1)}$ blocks can be arranged in a $\frac{v}{k} \times
 \frac{\lambda (v-1)}{k-1}$ array, and each block can be colored with one of $c$
 colors so that
 \begin{enumerate}[(i)]
 \item each point appears exactly once in each column, and
 \item in each row, blocks of the same color are pairwise disjoint.
 \end{enumerate}
 \end{definition}

\begin{definition}
A $\gbtd_\lambda(k,m)$ is {\em $c$-$*$colorable} if each of its
blocks can be colored with one of $c$ colors so that in each row,
blocks of the same color are pairwise disjoint.
\end{definition}

\begin{definition}
 A $k$-$*$colorable ${\rm RBIBD}(v,k,1)$ is {\em $k$-$*$colorable
 with property $\Pi$} if there exists a row $r$ such that for each
 color $i$, there exists a point (called a {\em witness} for $i$)
 that 
 is not contained in any block in row $r$ that is colored $i$.
 \end{definition}

 A $\gbtd_1(k,m)$ that is {\em $c$-$*$colorable with property $\Pi$} is similarly defined.
 
 \begin{figure*}[!t]
\renewcommand{\arraystretch}{1.2}
\centering
\begin{tabular}{|c|c|c|c|c|c|c|}
\hline
$0_0 0_1 \infty ~ \clubsuit$ & $2_0 4_0 3_1 ~ \clubsuit$ & $ 6_1 4_0 1_0  ~
\diamondsuit$ & $ 2_1 1_1 6_1 ~ \clubsuit$ & $ 5_1 1_0 2_0 ~ \heartsuit$ & $ 4_1
3_1 1_1 ~ \heartsuit$ & $ 5_1 4_1 2_1 ~ \diamondsuit$ \\ \hline

$ 6_1 5_1 3_1 ~ \clubsuit$ & $ 1_0 1_1 \infty ~ \clubsuit$ & $ 3_0 5_0 4_1 ~
\clubsuit$ & $ 3_0 3_1 \infty ~ \diamondsuit$ & $ 5_0 0_0 6_1 ~ \diamondsuit$ & $
6_0 1_0 0_1  ~ \diamondsuit$ & $ 0_0 2_0 1_1 ~ \heartsuit$ \\ \hline

$ 1_0 3_0 2_1 ~ \clubsuit$ & $ 0_1 6_1 4_1 ~ \clubsuit$ & $ 2_0 2_1 \infty ~
\diamondsuit$ & $ 4_0 6_0 5_1 ~ \clubsuit$ & $ 1_1 6_0 3_0 ~ \diamondsuit$ & $ 5_0
5_1 \infty ~ \heartsuit$ & $ 3_1 1_0 5_0  ~ \diamondsuit$ \\ \hline

$ 4_1 2_0 6_0 ~ \clubsuit$ & $ 5_1 3_0 0_0 ~ \clubsuit$ & $ 1_1 0_1 5_1 ~
\diamondsuit$ & $ 0_1 5_0 2_0 ~ \heartsuit$ & $ 4_0 4_1 \infty ~ \diamondsuit$ & $
2_1 0_0 4_0 ~ \heartsuit$ & $ 6_0 6_1 \infty ~ \heartsuit$ \\ \hline

$ 1_1 4_0 5_0 ~ \clubsuit$ & $ 2_1 5_0 6_0  ~ \diamondsuit$ & $ 3_1 6_0 0_0 ~
\clubsuit$ & $ 4_1 0_0 1_0 ~ \diamondsuit$ & $ 3_1 2_1 0_1 ~ \heartsuit$ & $ 6_1
2_0 3_0 ~ \clubsuit$ & $ 0_1 3_0 4_0 ~ \diamondsuit$ \\ \hline
\end{tabular}
\caption{A $3$-$*$colorable $\rbibd(15,3,1)$ $(X,\A)$, where $X=(\ZZ_7\times\ZZ_2)\cup\{\infty\}$. The set of colors
used is $\{\clubsuit,\diamondsuit,\heartsuit\}$. $(X,\A)$ has property $\Pi$ as $1_0$ is a witness for $\clubsuit$ and $\infty$ is a witnesses for both $\diamondsuit$ and $\heartsuit$ in row $1$. For succintness, a block $\{x,y,z\}$ is written $xyz$} 
\label{fig:rbibd15}
\end{figure*}

\begin{example}
\label{eg:5}
The $\rbibd(15,3,1)$ in Fig. \ref{fig:rbibd15} is $3$-$*$colorable with property $\Pi$.
\end{example}

\begin{proposition}
\label{lem:color} If an ${\rm RBIBD}(v,k,1)$ is
$(k-1)$-$*$colorable, then it is $k$-$*$colorable with property
$\Pi$.
\end{proposition}

 \begin{proof}
 Consider a $(k-1)$-$*$colorable $\rbibd(v,k,1)$ with colors
 $c_1,c_2,\cdots,c_{k-1}$. There must exists a point, say $x$, that
 appears only once in the first row. Recolor the block that contains
 this point with color $c_k$. This new coloring shows that the
 $\rbibd(v,k,1)$ is $k$-$*$colorable with property $\Pi$, 
since for the first row, the point $x$ is a witness for colors $c_1,c_2,\ldots,c_{k-1}$,
and any point not in the block colored by $c_k$ is a witness for $c_k$.
 \end{proof}

\begin{example}
The ${\rm GBTD}_1(3,9)$ in Fig. \ref{fig:gbtd27} is
$2$-$*$colorable and is therefore $3$-$*$colorable
with property $\Pi$ by Proposition \ref{lem:color}.
\end{example}

We note that a $3$-$*$colorable RBIBD and a $3$-$*$colorable RBIBD with property $\Pi$ 
are crucial in the tripling construction of a GBTD$_1(3,m)$ and a special GBTD$(_1(3,m)$  respectively
(see Proposition \ref{prop:tripling}). 
This is an adaptation of the $k$-tupling construction for GBTDs with index $k-1$ \cite[Theorem 3.1]{Lamken:1994}. 
However, we note certain differences. An FGBTD by definition is necessary a GBTD,
while $*$-colorability and property $\Pi$ are defined for RBIBDs.
Hence, we do not need a smaller GBTD to seed the recursion in Proposition \ref{prop:tripling}.
We make use of this fact to yield a special GBTD$_1(3,15)$ in Lemma \ref{lem:3^rq}.

 \subsection{Incomplete Generalized Balanced Tournament Packings}
 \label{sec:igbtp}

Incomplete designs are ubiquitous in combinatorial design theory and 
crucial in ``filling in the holes'' constructions described in section \ref{sec:recursive}.

 Suppose that $(X,\A)$ is a $(v,K,\lambda)$-packing. Let $W\subset X$ with $|W|=w$.
 Furthermore, we call $(X,W,\A)$ is an {\em incomplete resolvable packing},
 denoted by IRP$(v,K,\lambda;w)$, if it satisfies the following conditions:
 \begin{enumerate}[(i)]
 \item any pair of points from
 $W$ occurs in no blocks of $\A$,
 \item the blocks in $\A$ can be partitioned
 into parallel classes and {\em partial parallel classes} $X\setminus W$.
 \end{enumerate}

 \begin{definition}\label{defn:igbtp}
 Let $(X,W,\A)$ be an IRP$(v,K,\lambda;w)$.
 Then $(X,W,\A)$ is called an {\em incomplete generalized balanced tournament
 packing} (IGBTP)
 if the blocks of $\A$ are arranged into an $m\times n$ array $\sf A$, with rows and columns 
 indexed by $R$ and $C$ respectively,
 satisfying the following conditions:
 \begin{enumerate}[(i)]
 \item there exist a $P\subset R$ with $|P|=m'$ and a $Q\subset C$ with $|Q|=n'$ such that the cell $(r,c)$ is empty
 if $r\in P$ and $c\in Q$;
 \item for any row $r\in P$, every point in $X\setminus W$ is contained in
 either  $\ceiling{n/m}$ or $\floor{n/m}$ cells and the points
 in $W$ do not appear; for any row $r\in R\setminus P$, every
 point in $X$ is contained in either  $\ceiling{n/m}$ or $\floor{n/m}$
 cells;
 \item the blocks in any column $c\in Q$ form a partial parallel class of $X\setminus W$ and the blocks in any column
 $c\in C\setminus Q$ forms a parallel class of $X$.
 \end{enumerate}
 
 Denote such an IGBTP by IGBTP$_\lambda(K,v,m\times n;w,m'\times n')$.
 \end{definition}

 \begin{example}
 \label{eg:igbtp} An ${\rm IGBTP}_1(\{2,3^*\},29, 14\times 25; 9,4\times 5)$ is given in Fig.
 \ref{fig:igbtp}.
 \end{example}

  \begin{figure}[t!]
  \renewcommand{\arraystretch}{2}
 \setlength{\tabcolsep}{10pt} \centering \Large
\begin{tabular}{|c|c|}
\hline 
$\vA$ & $\vB$\\
\hline
\end{tabular}

\vskip 5pt
\normalsize \flushleft where $\vA$ is the array
\vskip 5pt

\renewcommand{\arraystretch}{1.2}
 \setlength\tabcolsep{1pt}
 \centering
\begin{tabular}{|r|r|r| r|r|r|r| }
\hline
$1_0 7_1 \infty_2 ~ \clubsuit$ &
$6_0 3_2 \infty_1 ~ \clubsuit$ &
$0_0 4_1 6_2 ~ \clubsuit$ &
$6_0 7_1 0_2 ~ \diamondsuit$ &
$7_0 0_1 1_2 ~ \clubsuit$ &
$5_0 5_1 5_2 ~ \diamondsuit$ &
$1_1 4_2 \infty_0 ~ \clubsuit$ \\\hline
$5_0 2_2 \infty_1 ~ \diamondsuit$ &
$2_0 0_1 \infty_2 ~ \clubsuit$ &
$7_0 4_2 \infty_1 ~ \clubsuit$ &
$1_0 5_1 7_2 ~ \clubsuit$ &
$4_0 4_1 4_2 ~ \diamondsuit$ &
$0_0 1_1 2_2 ~ \clubsuit$ &
$7_0 5_1 \infty_2 ~ \diamondsuit$ \\\hline
$3_1 6_2 \infty_0 ~ \clubsuit$ &
$4_1 7_2 \infty_0 ~ \diamondsuit$ &
$3_0 1_1 \infty_2 ~ \clubsuit$ &
$0_0 5_2 \infty_1 ~ \clubsuit$ &
$2_0 6_1 0_2 ~ \clubsuit$ &
$3_0 7_1 1_2 ~ \diamondsuit$ &
$1_0 2_1 3_2 ~ \clubsuit$ \\\hline
$3_0 4_1 5_2 ~ \clubsuit$ &
$1_0 1_1 1_2 ~ \diamondsuit$ &
$5_1 0_2 \infty_0 ~ \clubsuit$ &
$4_0 2_1 \infty_2 ~ \diamondsuit$ &
$5_0 3_1 \infty_2 ~ \clubsuit$ &
$2_0 7_2 \infty_1 ~ \clubsuit$ &
$4_0 0_1 2_2 ~ \clubsuit$ \\\hline
$6_0 2_1 4_2 ~ \diamondsuit$ &
$4_0 5_1 6_2 ~ \clubsuit$ &
$5_0 6_1 7_2 ~ \diamondsuit$ &
$6_1 1_2 \infty_0 ~ \clubsuit$ &
$1_0 6_2 \infty_1 ~ \diamondsuit$ &
$6_0 4_1 \infty_2 ~ \clubsuit$ &
$3_0 0_2 \infty_1 ~ \clubsuit$ \\\hline
$0_0 0_1 0_2 ~ \diamondsuit$ &
$5_0 2_1 0_2 ~ \clubsuit$ &
$4_0 7_1 3_2 ~ \clubsuit$ &
$2_0 1_1 6_2 ~ \clubsuit$ &
$7_1 2_2 \infty_0 ~ \diamondsuit$ &
$1_0 6_1 4_2 ~ \clubsuit$ &
$0_0 3_1 7_2 ~ \clubsuit$ \\\hline
$7_0 6_1 3_2 ~ \clubsuit$ &
$7_0 3_1 5_2 ~ \diamondsuit$ &
$6_0 3_1 1_2 ~ \clubsuit$ &
$5_0 0_1 4_2 ~ \clubsuit$ &
$3_0 2_1 7_2 ~ \clubsuit$ &
$0_1 3_2 \infty_0 ~ \diamondsuit$ &
$2_0 7_1 5_2 ~ \clubsuit$ \\\hline
$2_0 5_1 1_2 ~ \clubsuit$ &
$0_0 7_1 4_2 ~ \clubsuit$ &
$2_0 2_1 2_2 ~ \diamondsuit$ &
$7_0 4_1 2_2 ~ \clubsuit$ &
$6_0 1_1 5_2 ~ \clubsuit$ &
$4_0 3_1 0_2 ~ \clubsuit$ &
$6_0 6_1 6_2 ~ \diamondsuit$ \\\hline
$4_0 1_1 7_2 ~ \clubsuit$ &
$3_0 6_1 2_2 ~ \clubsuit$ &
$1_0 0_1 5_2 ~ \clubsuit$ &
$3_0 3_1 3_2 ~ \diamondsuit$ &
$0_0 5_1 3_2 ~ \clubsuit$ &
$7_0 2_1 6_2 ~ \clubsuit$ &
$5_0 4_1 1_2 ~ \clubsuit$ \\\hline
\end{tabular}

\vskip 5pt
\normalsize \flushleft where $\vB$ is the array
\vskip 5pt

\renewcommand{\arraystretch}{1.2}
 \setlength\tabcolsep{1pt}
 \centering
\begin{tabular}{|r|r|r| r|r|r|r| }
\hline
$0_0 6_1 \infty_2 ~ \diamondsuit$ &
$4_0 6_1 5_2 ~ \clubsuit$ &
$1_0 2_0 4_0 ~ \diamondsuit$ &
$2_0 3_0 5_0 ~ \clubsuit$ &
$2_2 7_2 0_2 ~ \clubsuit$ &
$2_1 3_1 5_1 ~ \clubsuit$ \\\hline
$2_1 5_2 \infty_0 ~ \clubsuit$ &
$5_0 7_1 6_2 ~ \clubsuit$ &
$3_1 4_1 6_1 ~ \clubsuit$ &
$0_1 1_1 3_1 ~ \diamondsuit$ &
$3_0 4_0 6_0 ~ \clubsuit$ &
$3_2 0_2 1_2 ~ \clubsuit$ \\\hline
$2_0 3_1 4_2 ~ \diamondsuit$ &
$6_0 0_1 7_2 ~ \clubsuit$ &
$4_2 1_2 2_2 ~ \clubsuit$ &
$4_1 5_1 7_1 ~ \clubsuit$ &
$5_1 6_1 0_1 ~ \diamondsuit$ &
$4_0 5_0 7_0 ~ \clubsuit$ \\\hline
$7_0 7_1 7_2 ~ \diamondsuit$ &
$0_0 2_1 1_2 ~ \clubsuit$ &
$0_2 5_2 6_2 ~ \diamondsuit$ &
$6_0 7_0 1_0 ~ \clubsuit$ &
$6_2 3_2 4_2 ~ \clubsuit$ &
$6_1 7_1 1_1 ~ \clubsuit$ \\\hline
$5_0 1_1 3_2 ~ \clubsuit$ &
$1_0 3_1 2_2 ~ \clubsuit$ &
$7_1 0_1 2_1 ~ \clubsuit$ &
$5_2 2_2 3_2 ~ \diamondsuit$ &
$7_0 0_0 2_0 ~ \clubsuit$ &
$7_2 4_2 5_2 ~ \clubsuit$ \\\hline
$6_0 5_1 2_2 ~ \clubsuit$ &
$2_0 4_1 3_2 ~ \diamondsuit$ &
$3_0 7_0 \infty_0 ~ \clubsuit$ &
$1_2 6_2 7_2 ~ \diamondsuit$ &
$5_2 1_2 \infty_2 ~ \clubsuit$ &
$4_1 0_1 \infty_1 ~ \clubsuit$ \\\hline
$1_0 4_1 0_2 ~ \clubsuit$ &
$3_0 5_1 4_2 ~ \diamondsuit$ &
$1_1 5_1 \infty_1 ~ \clubsuit$ &
$4_0 0_0 \infty_0 ~ \clubsuit$ &
$1_1 2_1 4_1 ~ \diamondsuit$ &
$6_2 2_2 \infty_2 ~ \clubsuit$ \\\hline
$3_0 0_1 6_2 ~ \clubsuit$ &
$\infty_0 \infty_1 \infty_2 ~ \diamondsuit$ &
$3_2 7_2 \infty_2 ~ \clubsuit$ &
$2_1 6_1 \infty_1 ~ \clubsuit$ &
$5_0 1_0 \infty_0 ~ \clubsuit$ &
$0_0 1_0 3_0 ~ \diamondsuit$ \\\hline
$4_0 1_2 \infty_1 ~ \diamondsuit$ &
$7_0 1_1 0_2 ~ \diamondsuit$ &
$5_0 6_0 0_0 ~ \diamondsuit$ &
$4_2 0_2 \infty_2 ~ \clubsuit$ &
$3_1 7_1 \infty_1 ~ \clubsuit$ &
$6_0 2_0 \infty_0 ~ \clubsuit$ \\\hline
\end{tabular}

\caption{A $2$-$*$colorable special ${\rm GBTD}_1(3,9)$ $(X,\A)$,
where $X=(\ZZ_8\times\ZZ_3)\cup\{\infty_0,\infty_1,\infty_2\}$ and colors $\{\clubsuit,\diamondsuit\}$. 
The cell $(1,5)$, 
occupied by the block $7_00_11_2$, is special. 
For succinctness, a set $\{x,y,z\}$ is written $xyz$. 
} \label{fig:gbtd27}
\end{figure}

\begin{figure}[h!]
\renewcommand{\arraystretch}{2}
 \setlength{\tabcolsep}{10pt} \centering \Large
\begin{tabular}{|c|c|}
\hline 
$\vA$ & $\vB$\\
\hline
\end{tabular}

\vskip 5pt
\normalsize \flushleft where $\vA$ is the array
\vskip 5pt

\renewcommand{\arraystretch}{1.2}
 \setlength\tabcolsep{1pt}
 \centering
 \small
\begin{tabular}{|c|c|c| c|c|c| c|c|c| c|c|c| c|c|c| c|c|c| }
\hline
-- & -- & -- & -- & -- & $2,13$ & $3,14$ & $4,15$ & $5,16$ & $6,17$ & $7,18$ & $8,19$ & $9,0$ \\\hline
-- & -- & -- & -- & -- & $12,16$ & $13,17$ & $14,18$ & $15,19$ & $16,0$ & $17,1$ & $18,2$ & $19,3$ \\\hline
-- & -- & -- & -- & -- & $15,18$ & $16,19$ & $17,0$ & $18,1$ & $19,2$ & $0,3$ & $1,4$ & $2,5$ \\\hline
-- & -- & -- & -- & -- & $1,3$ & $2,4$ & $3,5$ & $4,6$ & $5,7$ & $6,8$ & $7,9$ & $8,10$ \\\hline
$0,10$ & $2,7$ & $12,17$ & $4,16$ & $14,6$ & $4,5,11$ & $i,18$ & $h,1$ & $g,12$ & $f,18$ & $e,13$ & $d,16$ & $c,13$ \\\hline
$1,11$ & $3,8$ & $13,18$ & $5,17$ & $15,7$ & $a,0$ & $5,6,12$ & $i,19$ & $h,2$ & $g,13$ & $f,19$ & $e,14$ & $d,17$ \\\hline
$2,12$ & $4,9$ & $14,19$ & $6,18$ & $16,8$ & $b,7$ & $a,1$ & $6,7,13$ & $i,0$ & $h,3$ & $g,14$ & $f,0$ & $e,15$ \\\hline
$3,13$ & $5,10$ & $15,0$ & $7,19$ & $17,9$ & $c,6$ & $b,8$ & $a,2$ & $7,8,14$ & $i,1$ & $h,4$ & $g,15$ & $f,1$ \\\hline
$4,14$ & $6,11$ & $16,1$ & $8,0$ & $18,10$ & $d,10$ & $c,7$ & $b,9$ & $a,3$ & $8,9,15$ & $i,2$ & $h,5$ & $g,16$ \\\hline
$5,15$ & $7,12$ & $17,2$ & $9,1$ & $19,11$ & $e,8$ & $d,11$ & $c,8$ & $b,10$ & $a,4$ & $9,10,16$ & $i,3$ & $h,6$ \\\hline
$6,16$ & $8,13$ & $18,3$ & $10,2$ & $0,12$ & $f,14$ & $e,9$ & $d,12$ & $c,9$ & $b,11$ & $a,5$ & $10,11,17$ & $i,4$ \\\hline
$7,17$ & $9,14$ & $19,4$ & $11,3$ & $1,13$ & $g,9$ & $f,15$ & $e,10$ & $d,13$ & $c,10$ & $b,12$ & $a,6$ & $11,12,18$ \\\hline
$8,18$ & $10,15$ & $0,5$ & $12,4$ & $2,14$ & $h,19$ & $g,10$ & $f,16$ & $e,11$ & $d,14$ & $c,11$ & $b,13$ & $a,7$ \\\hline
$9,19$ & $11,16$ & $1,6$ & $13,5$ & $3,15$ & $i,17$ & $h,0$ & $g,11$ & $f,17$ & $e,12$ & $d,15$ & $c,12$ & $b,14$ \\\hline
\end{tabular}

\vskip 5pt
\normalsize \flushleft where $\vB$ is the array
\vskip 5pt

\renewcommand{\arraystretch}{1.2}
 \setlength\tabcolsep{1pt}
 \centering
 \small
\begin{tabular}{|c|c|c|c|c|c|c|c|c|c|c|c|c|c|c|c|c|c| }
\hline
$10,1$ & $11,2$ & $12,3$ & $13,4$ & $14,5$ & $15,6$ & $16,7$ & $17,8$ & $18,9$ & $19,10$ & $0,11$ & $1,12$ \\\hline
$0,4$ & $1,5$ & $2,6$ & $3,7$ & $4,8$ & $5,9$ & $6,10$ & $7,11$ & $8,12$ & $9,13$ & $10,14$ & $11,15$ \\\hline
$3,6$ & $4,7$ & $5,8$ & $6,9$ & $7,10$ & $8,11$ & $9,12$ & $10,13$ & $11,14$ & $12,15$ & $13,16$ & $14,17$ \\\hline
$9,11$ & $10,12$ & $11,13$ & $12,14$ & $13,15$ & $14,16$ & $15,17$ & $16,18$ & $17,19$ & $18,0$ & $19,1$ & $0,2$ \\\hline
$b,15$ & $a,9$ & $14,15,1$ & $i,8$ & $h,11$ & $g,2$ & $f,8$ & $e,3$ & $d,6$ & $c,3$ & $b,5$ & $a,19$ \\\hline
$c,14$ & $b,16$ & $a,10$ & $15,16,2$ & $i,9$ & $h,12$ & $g,3$ & $f,9$ & $e,4$ & $d,7$ & $c,4$ & $b,6$ \\\hline
$d,18$ & $c,15$ & $b,17$ & $a,11$ & $16,17,3$ & $i,10$ & $h,13$ & $g,4$ & $f,10$ & $e,5$ & $d,8$ & $c,5$ \\\hline
$e,16$ & $d,19$ & $c,16$ & $b,18$ & $a,12$ & $17,18,4$ & $i,11$ & $h,14$ & $g,5$ & $f,11$ & $e,6$ & $d,9$ \\\hline
$f,2$ & $e,17$ & $d,0$ & $c,17$ & $b,19$ & $a,13$ & $18,19,5$ & $i,12$ & $h,15$ & $g,6$ & $f,12$ & $e,7$ \\\hline
$g,17$ & $f,3$ & $e,18$ & $d,1$ & $c,18$ & $b,0$ & $a,14$ & $19,0,6$ & $i,13$ & $h,16$ & $g,7$ & $f,13$ \\\hline
$h,7$ & $g,18$ & $f,4$ & $e,19$ & $d,2$ & $c,19$ & $b,1$ & $a,15$ & $0,1,7$ & $i,14$ & $h,17$ & $g,8$ \\\hline
$i,5$ & $h,8$ & $g,19$ & $f,5$ & $e,0$ & $d,3$ & $c,0$ & $b,2$ & $a,16$ & $1,2,8$ & $i,15$ & $h,18$ \\\hline
$12,13,19$ & $i,6$ & $h,9$ & $g,0$ & $f,6$ & $e,1$ & $d,4$ & $c,1$ & $b,3$ & $a,17$ & $2,3,9$ & $i,16$ \\\hline
$a,8$ & $13,14,0$ & $i,7$ & $h,10$ & $g,1$ & $f,7$ & $e,2$ & $d,5$ & $c,2$ & $b,4$ & $a,18$ & $3,4,10$ \\\hline
\end{tabular}
 \caption{An ${\rm IGBTP}_1(\{2,3\},29, 14\times 25; 9,4\times 5)$ $(X,\A)$, where
 $X=\ZZ_{20}\cup\{a,b,c,d,e,f,g,h,i\}$ and $W=\{a,b,c,d,e,f,g,h,i\}$.
 For succinctness, a block $\{x,y,z\}$ is written $x,y,z$. 
 }
 \label{fig:igbtp}
 \end{figure}
 
\begin{figure}[h!]
\renewcommand{\arraystretch}{2}
 \setlength{\tabcolsep}{10pt} \centering \Large
\begin{tabular}{|c|c|}
\hline 
$\vA$ & $\vB$\\
\hline
\end{tabular}

\vskip 5pt
\normalsize \flushleft where $\vA$ is the array
\vskip 5pt

\renewcommand{\arraystretch}{1.2}
 \setlength\tabcolsep{1pt}
 \centering
 \small
\begin{tabular}{|c|c|c| c|c|c| c|c|c| c|c|c| c|c|c| c|c|c| }
\hline
-- & -- & -- & $4_0 1_0 7_0$ & $4_1 1_1 7_1$ & $4_2 1_2 7_2$ & $6_0 3_0 9_0$ & $6_1 3_1 9_1$ & $6_2 3_2 9_2$ \\\hline
-- & -- & -- & $6_0 7_2 8_0$ & $6_1 7_0 8_1$ & $6_2 7_1 8_2$ & $8_0 9_2 0_0$ & $8_1 9_0 0_1$ & $8_2 9_1 0_2$ \\\hline
$2_0 8_1 1_0$ & $2_1 8_2 1_1$ & $2_2 8_0 1_2$ & -- & -- & -- & $4_1 8_2 \infty_4$ & $4_2 8_0 \infty_3$ & $4_0 8_1 \infty_5$ \\\hline
$6_2 7_2 3_1$ & $6_0 7_0 3_2$ & $6_1 7_1 3_0$ & -- & -- & -- & $9_1 1_2 \infty_5$ & $9_2 1_0 \infty_4$ & $9_0 1_1 \infty_3$ \\\hline
$4_0 0_1 3_0$ & $4_1 0_2 3_1$ & $4_2 0_0 3_2$ & $1_1 3_0 9_1$ & $1_2 3_1 9_2$ & $1_0 3_2 9_0$ & -- & -- & -- \\\hline
$8_2 9_2 5_1$ & $8_0 9_0 5_2$ & $8_1 9_1 5_0$ & $4_2 6_1 8_2$ & $4_0 6_2 8_0$ & $4_1 6_0 8_1$ & -- & -- & -- \\\hline
$6_0 2_1 5_0$ & $6_1 2_2 5_1$ & $6_2 2_0 5_2$ & $8_1 1_2 \infty_0$ & $8_2 1_0 \infty_1$ & $8_0 1_1 \infty_2$ & $3_1 5_0 1_1$ & $3_2 5_1 1_2$ & $3_0 5_2 1_0$ \\\hline
$0_2 1_2 7_1$ & $0_0 1_0 7_2$ & $0_1 1_1 7_0$ & $2_0 3_1 \infty_1$ & $2_1 3_2 \infty_2$ & $2_2 3_0 \infty_0$ & $6_2 8_1 0_2$ & $6_0 8_2 0_0$ & $6_1 8_0 0_1$ \\\hline
$8_0 4_1 7_0$ & $8_1 4_2 7_1$ & $8_2 4_0 7_2$ & $2_2 9_0 \infty_2$ & $2_0 9_1 \infty_0$ & $2_1 9_2 \infty_1$ & $0_1 3_2 \infty_0$ & $0_2 3_0 \infty_1$ & $0_0 3_1 \infty_2$ \\\hline
$2_2 3_2 9_1$ & $2_0 3_0 9_2$ & $2_1 3_1 9_0$ & $3_2 4_1 \infty_3$ & $3_0 4_2 \infty_5$ & $3_1 4_0 \infty_4$ & $4_0 5_1 \infty_1$ & $4_1 5_2 \infty_2$ & $4_2 5_0 \infty_0$ \\\hline
$0_0 6_1 9_0$ & $0_1 6_2 9_1$ & $0_2 6_0 9_2$ & $2_1 6_2 \infty_4$ & $2_2 6_0 \infty_3$ & $2_0 6_1 \infty_5$ & $4_2 1_0 \infty_2$ & $4_0 1_1 \infty_0$ & $4_1 1_2 \infty_1$ \\\hline
$4_2 5_2 1_1$ & $4_0 5_0 1_2$ & $4_1 5_1 1_0$ & $7_1 9_2 \infty_5$ & $7_2 9_0 \infty_4$ & $7_0 9_1 \infty_3$ & $5_2 6_1 \infty_3$ & $5_0 6_2 \infty_5$ & $5_1 6_0 \infty_4$ \\\hline

\end{tabular}

\vskip 5pt
\normalsize \flushleft where $\vB$ is the array
\vskip 5pt

\renewcommand{\arraystretch}{1.2}
 \setlength\tabcolsep{1pt}
 \centering
 \small
\begin{tabular}{|c|c|c|c|c|c|c|c|c|c|c|c|c|c|c|c|c|c| }
\hline
$8_0 5_0 1_0$ & $8_1 5_1 1_1$ & $8_2 5_2 1_2$ & $0_0 7_0 3_0$ & $0_1 7_1 3_1$ & $0_2 7_2 3_2$ & $2_0 9_0 5_0$ & $2_1 9_1 5_1$ & $2_2 9_2 5_2$ \\\hline
$0_0 1_2 2_0$ & $0_1 1_0 2_1$ & $0_2 1_1 2_2$ & $2_0 3_2 4_0$ & $2_1 3_0 4_1$ & $2_2 3_1 4_2$ & $4_0 5_2 6_0$ & $4_1 5_0 6_1$ & $4_2 5_1 6_2$ \\\hline
$6_2 3_0 \infty_2$ & $6_0 3_1 \infty_0$ & $6_1 3_2 \infty_1$ & $4_1 7_2 \infty_0$ & $4_2 7_0 \infty_1$ & $4_0 7_1 \infty_2$ & $9_1 1_0 7_1$ & $9_2 1_1 7_2$ & $9_0 1_2 7_0$ \\\hline
$7_2 8_1 \infty_3$ & $7_0 8_2 \infty_5$ & $7_1 8_0 \infty_4$ & $8_0 9_1 \infty_1$ & $8_1 9_2 \infty_2$ & $8_2 9_0 \infty_0$ & $2_2 4_1 6_2$ & $2_0 4_2 6_0$ & $2_1 4_0 6_1$ \\\hline
$6_1 0_2 \infty_4$ & $6_2 0_0 \infty_3$ & $6_0 0_1 \infty_5$ & $8_2 5_0 \infty_2$ & $8_0 5_1 \infty_0$ & $8_1 5_2 \infty_1$ & $6_1 9_2 \infty_0$ & $6_2 9_0 \infty_1$ & $6_0 9_1 \infty_2$ \\\hline
$1_1 3_2 \infty_5$ & $1_2 3_0 \infty_4$ & $1_0 3_1 \infty_3$ & $9_2 0_1 \infty_3$ & $9_0 0_2 \infty_5$ & $9_1 0_0 \infty_4$ & $0_0 1_1 \infty_1$ & $0_1 1_2 \infty_2$ & $0_2 1_0 \infty_0$ \\\hline
-- & -- & -- & $8_1 2_2 \infty_4$ & $8_2 2_0 \infty_3$ & $8_0 2_1 \infty_5$ & $0_2 7_0 \infty_2$ & $0_0 7_1 \infty_0$ & $0_1 7_2 \infty_1$ \\\hline
-- & -- & -- & $3_1 5_2 \infty_5$ & $3_2 5_0 \infty_4$ & $3_0 5_1 \infty_3$ & $1_2 2_1 \infty_3$ & $1_0 2_2 \infty_5$ & $1_1 2_0 \infty_4$ \\\hline
$5_1 7_0 3_1$ & $5_2 7_1 3_2$ & $5_0 7_2 3_0$ & -- & -- & -- & $0_1 4_2 \infty_4$ & $0_2 4_0 \infty_3$ & $0_0 4_1 \infty_5$ \\\hline
$8_2 0_1 2_2$ & $8_0 0_2 2_0$ & $8_1 0_0 2_1$ & -- & -- & -- & $5_1 7_2 \infty_5$ & $5_2 7_0 \infty_4$ & $5_0 7_1 \infty_3$ \\\hline
$2_1 5_2 \infty_0$ & $2_2 5_0 \infty_1$ & $2_0 5_1 \infty_2$ & $7_1 9_0 5_1$ & $7_2 9_1 5_2$ & $7_0 9_2 5_0$ & -- & -- & -- \\\hline
$6_0 7_1 \infty_1$ & $6_1 7_2 \infty_2$ & $6_2 7_0 \infty_0$ & $0_2 2_1 4_2$ & $0_0 2_2 4_0$ & $0_1 2_0 4_1$ & -- & -- & -- \\\hline
\end{tabular}
\caption{An ${\rm FrGBTD}_1(3,6^6)$ $(X,\G,\A)$,
where $X=(\ZZ_{10}\times\ZZ_3)\cup\{\infty_i:i\in \ZZ_6\}$ and $\G= \{ \{t_0,t_1,t_2,(5+t)_0,(5+t)_1,(5+t)_2\}: t\in\ZZ_5\} \cup \{\infty_i: i\in\ZZ_6 \}$.
For succinctness, a set $\{x,y,z\}$ is written $xyz$.
 }

\label{fig:fgbtd6}
\end{figure}

 Consider an IGBTP$_1(\{k\},km,m\times \frac{km-1}{k-1};k,1\times 1)$. 
 Then its corresponding array has one empty cell and 
 we fill this cell with the block $W$ to obtain a GBTD$_1(k,m)$.
A GBTD$_1(k,m)$ obtained in this way is called a {\em special}
 GBTD$_1(k,m)$ and the cell occupied by $W$ is said to be {\em special}.

 \begin{example}
 \label{eg:9} The $\gbtd_1(3,9)$ in Fig. \ref{fig:gbtd27} is a
 special $\gbtd_1(3,9)$ with special cell $(1,5)$.
 \end{example}

 A few more classes of auxiliary designs are also required.

\subsection{Group Divisible Designs and Transversal Designs}

\begin{definition}
Let $(X,\A)$ be a set system and let $\G=\{G_1,G_2,\ldots,G_s\}$ be a partition of $X$ into
subsets, called {\em groups}. The triple $(X,\G,\A)$ is a {\em group divisible design} (GDD)
when every $2$-subset of $X$ not contained in a group appears in exactly one block,
and $|A\cap G|\leq 1$ for $A\in\A$ and $G\in\G$.
\end{definition}

We denote a GDD $(X,\G,\A)$ by $K$-GDD if $(X,\A)$ is $K$-uniform. The {\em type}
of a GDD $(X,\G,\A)$ is the multiset $\langle |G| : G\in\G\rangle$.
For convenience, the exponential
notation is used to describe the type of a GDD: a GDD of type $g_1^{t_1} g_2^{t_2} \cdots
g_s^{t_s}$ is a GDD with exactly $t_i$ groups of size $g_i$, $i\in[s]$.

\begin{definition}
A {\em transversal design} $\td(k,n)$ is a $\{k\}$-GDD of type $n^k$.
\end{definition}

The following result on the existence of transversal designs (see \cite{Abeletal:2007})
is sometimes used without explicit reference throughout this paper.
                                                       
\begin{theorem}
\label{thm:tdexistence}
Let ${\rm TD}(k)$ denote the set of positive integers $n$ such that there exists a ${\rm TD}(k,n)$.
Then, we have
\begin{enumerate}[(i)]
\item ${\rm TD}(4)\supseteq\ZZ_{>0}\setminus\{2,6\}$,
\item ${\rm TD}(5)\supseteq\ZZ_{>0}\setminus\{2,3,6,10\}$,
\item ${\rm TD}(6)\supseteq\ZZ_{>0}\setminus\{2,3,4,6,10,14,18,22\}$,
\item ${\rm TD}(7)\supseteq\ZZ_{>0}\setminus\{2, 3, 4, 5, 6, 10, 14, 15,18, 20, 22, 26, 30, 34, 38, 46, 60\}$,
\item ${\rm TD}(k)\supseteq\{q:\text{$q\geq k-1$ is a prime power}\}$.
\end{enumerate}
\end{theorem}
 
\begin{definition}
A {\em doubly resolvable} $\td(k,n)$, denoted by $\drtd(k,n)$, is a $\td(k,n)$ whose blocks can
be arranged in an $n\times n$ array such that each point appears exactly once in each row
and once in each column.
\end{definition}

The following proposition describes the relationship between DRTDs and TDs


 \begin{proposition}[Folklore, see \cite{Colbournetal:2004}]
 There exists a $\td(k+2,n)$ if and only if there exists a $\drtd(k,n)$.
 \end{proposition}

 \begin{corollary}
 \label{drtdexistence}
 A $\drtd(3,n)$ exists for all $n\geq 4$ and $n\not\in\{6,10\}$.
 \end{corollary}

 \begin{proof}
 A $\td(5,n)$ exists if $n\geq 4$ and $n\not\in\{6,10\}$ by Theorem \ref{thm:tdexistence}.
 \end{proof}

 \subsection{Frame Generalized Balanced Tournament Design}
 \label{sec:fgbtd}

 Let $(X,\G,\A$) be a $\{k\}$-GDD with $\G=\{G_1,G_2,\ldots,G_s\}$
 and $|G_i|\equiv 0\bmod k(k-1)$ for all $i\in [s]$. Let
 $R=\frac{1}{k}\sum_{i=1}^s |G_i|$ and $C=\frac{1}{k-1}\sum_{i=1}^s
 |G_i|$. Suppose there exists a partition $[R]=\bigsqcup_{i=1}^s R_i$
 and a partition $[C]=\bigsqcup_{i=1}^s C_i$ such that for each $i\in
 [s]$, we have $|R_i| = |G_i|/k$ and $|C_i|=|G_i|/(k-1)$.

 We say that $(X,\G,\A)$ is a
 {\em frame generalized balanced tournament design} (FrGBTD)
 if its blocks can be arranged in an $R\times C$ array
 such that the following conditions hold:
 \begin{enumerate}[(i)]
 \item the cell $(r,c)$ is empty when $(r,c)\in R_i\times C_i$ for $i\in[s]$,
 \item for any row $r\in R_i$, each point in $X\setminus G_i$ appears either once or twice and the
 points in $G_i$ do not appear,
 \item for any column $c\in C_i$, each point in $X\setminus G_i$ appears exactly once.
 \end{enumerate}
 Denote this FrGBTD by ${\rm FrGBTD}(k,T)$, where $T=\langle |G_i| :
 i\in[s] \rangle$.

 \begin{example}
\label{eg:fgbtd6} An ${\rm FrGBTD}(3,6^6)$ is given in Fig.
\ref{fig:fgbtd6}.
\end{example}

 \section{Recursive Constructions}
 \label{sec:recursive}

 In this section, we develop the necessary recursive constructions.
 We note that these are straightforward adaptions of methods in previous work
  \cite{Lamken:1992, Lamken:1994, Yinetal:2008, Daietal:2013}. 

 \subsection{Recursive Constructions for GBTPs}

  First, for block size three, we have the following tripling construction for GBTDs.
  This is an adaption of $k$-tupling construction for the case of GBTDs with index $k-1$ \cite[Theorem 3.1]{Lamken:1994}
  and the doubling construction for balanced tournament designs \cite{Schellenbergetal:1977}.
   
 \begin{proposition}[Tripling Construction]
 \label{prop:tripling}
 Suppose there exists a $3$-$*$colorable $\rbibd(m,3,1)$ and a $\drtd(3,m)$.
 Then there exists a $2$-$*$colorable $\gbtd_1(3,m)$. 
 Suppose further that the $\rbibd(m,3,1)$ is
 $3$-$*$colorable with property $\Pi$. Then the $\gbtd_1(3,m)$ is a special
 $\gbtd_1(3,m)$.
 \end{proposition}

\begin{proof}
Consider a $3$-$*$colorable $\rbibd(m,3,1)$ $(X,\A)$ with colors from $\ZZ_3$ and let
\begin{equation*}
X' = \{x_i : \text{$x\in X$ and $i\in\ZZ_3$}\}.
\end{equation*}
Make three copies of the $3$-$*$colorable $\rbibd(m,3,1)$ as follows: for the $j$th copy,
$j\in\{1,2,3\}$, each block $\{x,y,z\}$ of color $i$ in the
$3$-$*$colorable $\rbibd(m,3,1)$ is
replaced by block $\{x_{i+j},y_{i+j},z_{i+j}\}$, where arithmetic in the subscripts is
performed modulo three. Stacking these three $\frac{m}{3}\times\frac{m-1}{2}$
arrays together gives an $m\times \frac{m-1}{2}$ array $\vA$ with the property that
\begin{enumerate}[(i)]
\item each point in $X'$ appears exactly once in each column,
\item each point in $X'$ appears at most once in each row.
\end{enumerate}
Now take a $\drtd(3,m)$ $(X',\G,\A)$, where
\begin{equation*}
\G = \{ \{x_i: x\in X\} : i\in \ZZ_3\},
\end{equation*}
and adjoin it to $\vA$. This gives an $m\times \frac{3m-1}{2}$ array, which we claim
is a $\gbtd_1(3,m)$. Indeed it is easy to see that in this array, each point in $X'$ appears
exactly once in each column and either once or twice in each row. It remains to show that this
array is a $\bibd(3m,3,1)$. To see this, observe that any pair of points contained in a group
of the $\drtd(3,m)$ is contained in a block of one of the copies of the
$3$-$*$colorable $\rbibd(m,3,1)$. This $\gbtd_1(3,m)$ is
$2$-$*$colorable by giving the blocks from the $\drtd(3,m)$ one color and the remaining
blocks (from the three copies of the $\rbibd(m,3,1)$) another color.

If, in addition, the $\rbibd(m,3,1)$ is $3$-$*$colorable with property $\Pi$, and that
in row $r$ of this $\rbibd(m,3,1)$, the points $x,y,z$ (not necessarily distinct) are witnesses
for colors $0,1,2$, respectively, then we assume that the $\drtd(3,m)$ used has
the block $\{x_0,y_1,z_2\}$ and that this block can be made to appear in row $r$, by
permuting rows if necessary. The cell that contains $\{x_0,y_1,z_2\}$ is a special cell of the
$\gbtd_1(3,m)$.
\end{proof}

\begin{corollary}
\label{cor:3^km}
Let $m>3$ and
suppose  an $\rbibd(m,3,1)$ that is $3$-$*$colorable with property $\Pi$ exists.
Then there exists a special $\gbtd_1(3,3^k m)$, for all $k\geq 0$.
\end{corollary}

\begin{proof}
First note that $m\equiv 3\bmod 6$ since this is a necessary condition for the
existence of an $\rbibd(m,3,1)$. Hence, there exists a $\drtd(3,m)$ by
Corollary \ref{drtdexistence}.
By Proposition \ref{prop:tripling}, there exists a $2$-$*$colorable special
$\gbtd_1(3,m)$, which may be regarded as an $\rbibd(3m,3,1)$ that is
$3$-$*$colorable with property $\Pi$. The corollary then follows by induction.
\end{proof}

The following propositions are simple generalizations of the standard ``filling in the hole'' construction
to construct GBTPs or GBTDs using IGBTPs and FrGBTDs.

  
 \begin{proposition}[IGBTP Construction for GBTP]
 \label{prop:filling hole for IGBTP} 
 If an IGBTP$_\lambda(K,v,m\times n; w,m'\times n')$ 
 and a GBTP$_\lambda(K,w,m'\times n')$ exists,
 then a GBTP$_\lambda(K,v,m\times n)$ exists.
 \end{proposition}
 
\begin{proof}
Let $(X,\A)$	be an IGBTP$_\lambda(K,v,m\times n; w,m'\times n')$.	
Fill in the empty subarray of this IGBTP with an
a GBTP$_\lambda(K,w,m'\times n')$, $(X',\A')$.
The resulting array is a GBTP$_\lambda(K,v,m\times n)$, $(X,\A\cup\A')$.
\end{proof}

 FrGBTD is a useful tool to construct larger GBTPs from smaller ones.

  \begin{proposition}[FrGBTD Construction for GBTP]
 \label{fgbtd+IGBTP} Let $k\in K$. Suppose there exists an ${\rm FrGBTD}(k,T)$
 $(X,\G,\A)$, where $\G=\{G_1,G_2,\ldots,G_s\}$, and let
 $r_i=|G_i|/k$ and $c_i=|G_i|/(k-1)$, for $i\in[s]$.
 If there exists an IGBTP$_1(K,|G_i|+w,(r_i+m)\times (c_i+n);w,m\times n)$ for all $i\in[s]$, then there exists
 an IGBTP$_1(K,\sum_{i=1}^s|G_i|+w,(\sum_{i=1}^sr_i+m)\times (\sum_{i=1}^sc_i+n);w,m\times
 n)$. Furthermore, if a GBTP$_1(K,w,m\times n)$ exists,
 then an GBTP$_1(K,\sum_{i=1}^s|G_i|+w,(\sum_{i=1}^sr_i+m)\times
 (\sum_{i=1}^sc_i+n))$ exists.
 \end{proposition}
 \begin{proof}
 We use the notations as in the definition of FrGBTD in Section
 \ref{sec:fgbtd}, and assume that the blocks of the ${\rm
 FrGBTD}(k,T)$ are arranged in an $R\times C$ array, with rows and
 columns indexed by $[R]$ and $[C]$, respectively.

 Let $P$ and $Q$ be two sets satisfying $|P|=m, |Q|=n, P\cap [R]=\emptyset, Q\cap
 [C]=\emptyset$.

 For each $i\in[s]$, consider an IGBTP$_1(K,|G_i|+w,(r_i+m)\times (c_i+n);w,m\times n)$
 $(X_i,\A_i)$, where $X_i=G_i\cup\{\infty_1,\infty_2,\cdots,
 \infty_w\}$, and whose rows and columns are indexed by $P\cup
 R_i$ and $Q\cup C_i$, respectively.
 It can be verified that
 $(X',\A')$, where
 \begin{align*}
 X' &= X \cup \{\infty_1,\infty_2,\cdots,\infty_w\}, \\
 \A' &= \A \cup ( \cup_{i=1}^s \A_i),
 \end{align*}
 is an IRP$(\sum_{i=1}^s|G_i|+w,K,1)$.

 Arrange the blocks of $(X',\A')$ into an
 $(R+m')\times(C+n')$ array $\vA$, whose rows and columns are
 indexed by $P\cup [R]$ and $Q\cup [C]$, respectively, such that
 each block in $\A$ that appears in cell $(i,j)$ of either the FrGBTD or the IGBTP, is
 placed in cell $(i,j)$ of $\vA$.

 The definition conditions of an FrGBTD ensures that no cells are
 occupied by two blocks. It is also easily checked that every point
 in $X'$ appears exactly once in each column and either once or twice
 in each row. In addition, the $m\times n$ subarray indexed by $P\times
 Q$ is empty. This gives an IGBTP$_1(K,\sum_{i=1}^s|G_i|+w,(\sum_{i=1}^sr_i+m)\times (\sum_{i=1}^sc_i+n);w,m\times
 n)$.

 The last statement follows from Proposition \ref{prop:filling hole for IGBTP}.
 \end{proof}

 Since a GBTD is an instance of GBTP, we have the following recursive construction
for GBTDs.

 \begin{corollary}[FrGBTD Construction for GBTD]
 \label{fgbtdconstruction} Suppose an ${\rm FrGBTD}(k,T)$ exists
 with groups $\{G_1,G_2,\ldots,G_s\}$. Let
 $g_i=|G_i|/k$, for $i\in[s]$. If there exists a special
 $\gbtd_1(k,g_i+1)$ for all $i\in[s]$, then there exists a special
 $\gbtd_1(k,\sum_{i=1}^s g_i+1)$.
 \end{corollary}

 When the groups are of the same size, we have the following corollary.

\begin{corollary}
\label{gbtd:gt+1}
If there exists an ${\rm FrGBTD}(3,(3g)^t)$ and a special ${\rm GBTD}_1(3,g+1)$,
then there exists a special ${\rm GBTD}_1(3,gt+1)$.
\end{corollary}

For Proposition \ref{prop:filling hole for IGBTP} and Corollary
\ref{fgbtdconstruction} to be useful, we require large classes of
FrGBTDs. We give three recursive constructions for FrGBTDs next.

\subsection{Recursive Constructions for FrGBTDs}

We adapt the standard direct product construction.

\begin{proposition}[Inflation]
\label{prop:inflation} Suppose an ${\rm FrGBTD}(k,T)$
and a ${\rm DRTD}(k,n)$ exists. Then there exists an ${\rm FrGBTD}(k,nT)$.
\end{proposition}

\begin{proof}
Let $(X,\G,\A)$ be an ${\rm FrGBTD}(k,T)$ arranged in an $R\times C$
array $\vA$, with rows and columns indexed by $[R]$ and $[C]$,
respectively. Define
\begin{align*}
X' &= X\times [n], \\
\G' &= \{G\times [n]: G\in\G\},
\end{align*}
and for each block $A\in\A$, let
\begin{align*}
X_A &= A\times [n], \\
\G_A &= \{ \{x\}\times[n]: x \in A \}.
\end{align*}
and let $(X_A,\G_A,\B_A)$ be a ${\rm DRTD}(k,n)$ whose blocks are
arranged in an $n\times n$ array with rows and columns both indexed
by $[n]$. Let $\A'=\cup_{A\in\A}\B_A$ and the blocks in $\A'$ can be
arranged, as follows, in an $Rn\times Cn$ array, whose rows and
columns are indexed by $[R]\times n$ and $[C]\times n$,
respectively: a block $B\in\B_A$ is placed in cell $((i,a),(j,b))$
if $A$ appears in cell $(i,j)$ of the ${\rm FrGBTD}(k,T)$ and $B$
appears in cell $(a,b)$ of the ${\rm DRTD}(k,n)$. Hence,
$(X',\G',\A')$ gives an ${\rm FrGBTD}(k,nT)$.
\end{proof}
\vskip 10pt

Wilson's Fundamental Construction for GDDs \cite{Wilson:1972a} can also be modified to
construct FrGBTDs. Fig. \ref{WFC} describes this construction.

\begin{figure}[!t]
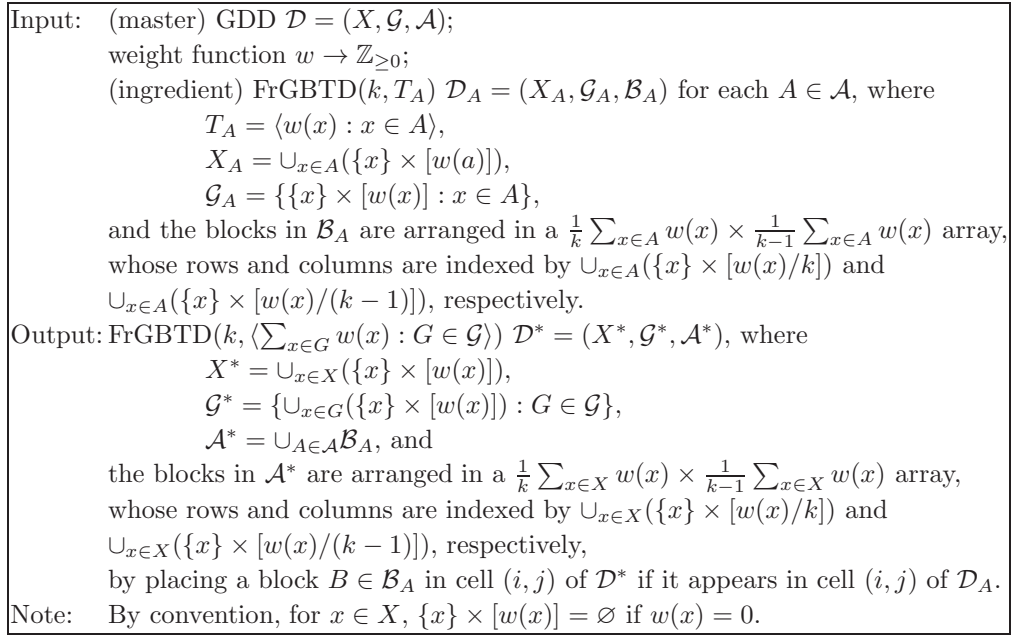

\renewcommand{\arraystretch}{1.1}
\setlength{\tabcolsep}{1pt} 

\centering
\small
\begin{tabular}{| l l|}
\hline
Input: & (master) GDD $\D=(X,\G,\A)$; \\
           & weight function $w\rightarrow\mathbb{Z}_{\geq 0}$; \\
           & (ingredient) ${\rm FrGBTD}(k,T_A)$ $\D_A=(X_A,\G_A,\B_A)$ for each $A\in\A$, where \\
           & ~~~~~~~~~~ $T_A=\langle w(x): x\in A\rangle$, \\
           & ~~~~~~~~~~ $X_A=\cup_{x\in A} (\{x\}\times[w(a)])$,  \\
           & ~~~~~~~~~~ $\G_A=\{ \{x\}\times[w(x)] : x\in A\}$, \\
           & and the blocks in $\B_A$ are arranged in a
           $\frac{1}{k}\sum_{x\in A}w(x) \times \frac{1}{k-1}\sum_{x\in A}w(x)$ array,\\
           & whose rows and columns are indexed by
           $\cup_{x\in A}(\{x\}\times[w(x)/k])$ and \\
           & $\cup_{x\in A}(\{x\}\times[w(x)/(k-1)])$, respectively. \\
Output: & ${\rm FrGBTD}(k,\langle \sum_{x\in G} w(x) :
G\in\G\rangle)$
$\D^*=(X^*,\G^*,\A^*)$, where \\
              & ~~~~~~~~~~ $X^*=\cup_{x\in X}(\{x\}\times [w(x)])$, \\
              & ~~~~~~~~~~ $\G^*=\{\cup_{x\in G}(\{x\}\times [w(x)]) : G\in\G\}$,  \\
              & ~~~~~~~~~~ $\A^*=\cup_{A\in\A}\B_A$, and \\
              & the blocks in $\A^*$ are arranged in a
              $\frac{1}{k}\sum_{x\in X} w(x) \times \frac{1}{k-1}\sum_{x\in X} w(x)$ array,\\ 
              & whose rows and columns are indexed by
              $\cup_{x\in X}(\{x\}\times [w(x)/k])$ and \\
              & $\cup_{x\in X}(\{x\}\times [w(x)/(k-1)])$, respectively, \\
              &by placing a block $B\in\B_A$ in cell $(i,j)$ of $\D^*$ if it appears in
              cell $(i,j)$ of $\D_A$. \\
Note: & By convention, for $x\in X$, $\{x\}\times[w(x)]=\varnothing$ if $w(x)=0$.\\
\hline
\end{tabular}
\caption{Fundamental Construction for FrGBTDs}
\label{WFC}
\end{figure}

 \begin{proposition}[Fundamental Construction]
 \label{FC} Suppose there exists a (master) GDD $(X,\G,\A)$ of type
 $T$ and let $w:X\to\ZZ_{\geq 0}$ be a weight function. If for each
 $A\in\A$, there exists an (ingredient) ${\rm FrGBTD}(k,\langle
 w(a):a\in A\rangle)$, then there exists an ${\rm FrGBTD}(k,\langle
 \sum_{x\in G} w(x) : G\in\G\rangle)$.
 \end{proposition}

\begin{proof}
The Fundamental Construction in Fig. \ref{WFC} constructs the desired FrGBTD
from the master GDD and ingredient FrGBTDs.
\end{proof}

Proposition \ref{FC} admits the following specialization.

\begin{proposition}[PBD Closure]
\label{prop:pbdclosure}
Suppose there exists a ${\rm PBD}(v,K)$ $(X,\A)$, and for each block $A\in\A$, there
exists an ${\rm FrGBTD}(3,g^{|A|})$. Then there exists an ${\rm FrGBTD}(3,g^v)$.
\end{proposition}

\begin{proof}
Consider the PBD as a (master) GDD of type $1^v$ and weight function $w(x)=g$ for all
$x\in X$. Now apply the fundamental construction.
\end{proof}

\begin{proposition}[FrGBTD from Truncated TD]
\label{prop:truncation} Let $s>0$. 
Suppose there exists a TD$(u+s,m)$,
and $g_1,g_2,\ldots,g_s$ are nonnegative integers at most $m$. 
If there exists an ${\rm FrGBTD}(k,g^t)$ for each $t\in\{u,u+1,\ldots,u+s\}$,
then there exists an ${\rm FrGBTD}(k,T)$, where $T=(g\cdot
m)^u(g\cdot g_1)(g\cdot g_2)\cdots (g\cdot g_s)$.
\end{proposition}

\begin{proof}
For each $i\in[s]$, delete $m-g_i$ points from the $i$th group of the ${\rm TD}(u+s,m)$.
This results in a $\{u,u+1,\ldots,u+s\}$-GDD of type $m^u g_1 g_2 \cdots g_s$. Use this
as the master GDD and apply the fundamental construction with weight function
$w$ that assigns weight $g$ to all points.
\end{proof}

\section{Direct Constructions}
\label{sec:direct}

This section constructs some small GBTDs and FrGBTDs that are required to seed
the recursive constructions given in Section \ref{sec:recursive}. 
Our main tools are {\em starters} 
and the {\em method of differences} .

Starter-adder constructions are ubiquitous in the constructions for GBTDs with index $k-1$,
associated frames and other types of similar designs 
(see for example, \cite{Lamken:1992,Lamken:1994,Yinetal:2008,Cheeetal:2013b,Daietal:2013}).
Unlike previous work and due to the lack of symmetry in our arrays,
we fix the positions of the starters in our arrays and 
`develop' the blocks in a variety of `directions' (see Figures \ref{fig:gbtd3m}, \ref{fig:GBTPstarter-Z2} and \ref{fig:GBTPstarter-Z4}).
This removes the use of adders and 
surprisingly a careful analysis of the starter conditions
allows a prime power construction that is given in Proposition \ref{prop:fqstarterforGBTD(3,m)}.

 First, we recall certain concepts with regards to the method of differences.
 Let $\Gamma$ be an additive abelian group and let $n$ be a positive integer.
 For a set system $(\Gamma, \SSS)$, the {\em difference list} of
 $\SSS$ is the multiset
 \begin{equation*}
 \Delta\SSS = \langle x-y : \text{$x,y\in A$, $x\not=y$, and
 $A\in\SSS$} \rangle.
 \end{equation*}
 For a set-system $(\Gamma\times[n], \SSS)$ and $i,j\in [n]$, the multiset
 \begin{equation*}
 \Delta_{ij}\SSS = \langle x-y : \text{$x_i,y_j\in A$, $x_i\not=y_j$,
 and $A\in\SSS$}\rangle
 \end{equation*}
 is called a list of {\em pure differences} when $i=j$, and called a
 list of {\em mixed differences} when $i\not=j$.

\subsection{Direct Constructions for GBTDs}


\begin{definition}[Starter for GBTD]
 \label{defn:starter} Let $m$ be an odd positive integer, $\Gamma$ be
 an additive abelian group of size $m$. Let $T$ be an index set of
 size $(m-1)/2$. Let $(\Gamma\times[3],\SSS)$ be a
 $\{3\}$-uniform set system of size $(3m-1)/2$, where
 \begin{equation*}
 \SSS =\{A_\alpha: \alpha\in\Gamma\} \cup \{B_t: t\in T\}.
 \end{equation*}

 $\SSS$ is called a {\em $(\Gamma\times[3])$-GBTD-starter} if the
 following conditions hold:
 \begin{enumerate}[(i)]
 \item $\Delta_{ii}\SSS=\Gamma\setminus\{0\}$, for $i\in[3]$,
 \item $\Delta_{ij}\SSS=\Gamma$, for $i,j\in [3]$, $i\not=j$,
 \item $\cup_{\alpha\in\Gamma} A_\alpha = \Gamma\times[3]$,
 \item $\{j: \text{$\alpha_j\in B_t$ for some $\alpha\in\Gamma$}\}=[3]$, for $t\in T$,
 \item each element in $\Gamma\times[3]$ appears either once or twice in the multiset
 \begin{equation*}
 R =  \left(\bigcup_{\alpha\in\Gamma} A_\alpha-\alpha\right) \cup
 \left( \bigcup_{t\in T} B_t \right).
 \end{equation*}\end{enumerate}
  Furthermore, $\SSS$ is said to be \emph{special} if
 \begin{enumerate}[(i)]
 \setcounter{enumi}{5}
 \item each element in $A_0$ appears exactly once in $R$. 
 \end{enumerate}
 Also, $\SSS$ is said to be {\em $3$-$*$colorable with property $\Pi$} if each of the blocks in
 \begin{equation*}
\setcounter{enumi}{6}
\{A_\alpha-\alpha:\alpha\in\Gamma\} {\rm\ and\ } \{B_t: t\in T\},
\end{equation*}
\noindent can be colored with one of three colors so that
\begin{enumerate}[(i)]
\setcounter{enumi}{6}
\item blocks of the same color are pairwise disjoint,
\item for each color $c$, there exists a point (a {\em witness} for $c$)
that is not contained in any block assigned color $c$.
\end{enumerate}
 \end{definition}

 \begin{figure*}
 \renewcommand{\arraystretch}{2}
 \setlength{\tabcolsep}{10pt} \centering \Large
\begin{tabular}{|c|c|}
\hline 
$\vA$ & $\vB$\\
\hline
\end{tabular}

\vskip 5pt
\normalsize \flushleft where $\vA$ is the array
\vskip 5pt

 \renewcommand{\arraystretch}{1.2}
 \setlength{\tabcolsep}{2pt} \centering \normalsize

 \begin{tabular}{|ccccc|}
 \hline $A_0$ & $A_{-\alpha_{1}}+{\alpha_{1}}$ &
 $A_{-\alpha_{2}}+\alpha_{2}$ & $\cdots$ &
 $A_{-\alpha_{m-1}}+\alpha_{m-1}$ \\
 $A_{\alpha_1}$ & $A_{0}+{\alpha_{1}}$ &
 $A_{\alpha_1-\alpha_{2}}+\alpha_{2}$ & $\cdots$ &
 $A_{\alpha_1-\alpha_{m-1}}+\alpha_{m-1}$ \\
 $\vdots$ & $\vdots$ & $\vdots$ & $\ddots$ & $\vdots$ \\
 $A_{\alpha_{m-1}}$ & $A_{\alpha_{m-1}-\alpha_1}+{\alpha_{1}}$ &
 $A_{\alpha_{m-1}-\alpha_{2}}+\alpha_{2}$ & $\cdots$ &
 $A_{0}+\alpha_{m-1}$ \\
 \hline
 \end{tabular}

\vskip 5pt
\normalsize \flushleft and $\vB$ is the array
\vskip 5pt

 \renewcommand{\arraystretch}{1.2}
 \setlength{\tabcolsep}{2pt} \centering \normalsize

 \begin{tabular}{|cccc|}
 \hline 
 
 $B_{1}$ & $B_{2}$  & $\cdots$ & $B_{(m-1)/(k-1)}$ \\
 $B_{1}+\alpha_1$ & $B_{2}+\alpha_1$  & $\cdots$ & $B_{(m-1)/(k-1)}+\alpha_1$ \\
 $\vdots$ & $\vdots$  & $\ddots$ & $\vdots$ \\
 $B_{1}+\alpha_{m-1}$ & $B_{2}+\alpha_{m-1}$ & $\cdots$ & $B_{(m-1)/(k-1)}+\alpha_{m-1}$ \\
 \hline
 \end{tabular}\ .

 \caption{A ${\rm GBTD}_1(k,m)$ from
 $(\Gamma\times[k])$-GBTD-starter $\SSS=\{A_\alpha: \alpha\in\Gamma\}
 \cup \{B_t: t\in T\}$, where
 $\Gamma=\{0,\alpha_1,\ldots,\alpha_{m-1}\}$ and $T=[(m-1)/(k-1)]$.}
 \label{fig:gbtd3m}
 \end{figure*}

 \begin{proposition}\label{prop:starter-gbtd}
 If a $(\Gamma\times[k])$-GBTD-starter exists, 
 then a ${\rm GBTD}_1(k,m)$ exists.
 Similarly, if there exists a special $(\Gamma\times[3])$-GBTD-starter,
  then there exists a special GBTD$_1(3,m)$;
 and if there exists a $3$-$*$colorable $(\Gamma\times[3])$-GBTD-starter
 with property $\Pi$, then there exists a $3$-$*$colorable
 GBTD$_1(3,m)$ with property $\Pi$.
 \end{proposition}
 \begin{proof}
 Let $X=\Gamma\times[k]$, and suppose $\SSS=\{A_\alpha:
 \alpha\in\Gamma\} \cup \{B_t: t\in T\}$ is an
 $(\Gamma\times[k])$-GBTD-starter. Let
 \begin{equation*}
 \A = \bigcup_{A\in\SSS} \{ A+\alpha: \alpha\in\Gamma\}.
 \end{equation*}
 Then $(X,\A)$ is a ${\rm BIBD}(km,k,1)$, whose blocks can be
 arranged in an $m\times \frac{(km-1)}{k-1}$ array, whose rows and columns are
 indexed by $\Gamma$ and $\Gamma\cup T$, respectively, as follows:
 \begin{itemize} 
 \item for $\alpha,\beta\in\Gamma$, the block $A_\alpha+\beta$ is
 placed in cell $(\alpha+\beta,\beta)$, and
 \item for $t\in T$ and $\alpha\in\Gamma$, the block $B_t+\alpha$ is placed in cell $(\alpha,t)$.
 \end{itemize}
 Fig. \ref{fig:gbtd3m} depicts the placement of blocks in the array.

 For $\beta\in\Gamma$, the set of blocks occupying column $\beta$ is
 $\{A_\alpha+\beta: \alpha\in\Gamma\}$, which form a resolution class
 by condition (iii) of Definition \ref{defn:starter}. Similarly, for
 $t\in T$, the set of blocks occupying column $t$ is
 $\{B_t+\alpha:\alpha\in\Gamma\}$, which form a resolution class by
 condition (iv) in Definition \ref{defn:starter}.

 The set of blocks occupying row $0$ is given by $R$, and by
 condition (v) of Definition \ref{defn:starter}, each point in $X$
 appears either once or twice in row $0$. Since the blocks occupying
 row $\alpha$ ($\alpha\in\Gamma$) are exactly the translates of the
 blocks in $R$ by $\alpha$, every point in $X$ also appears either
 once or twice in row $\alpha$.
 
Suppose $\SSS =\{A_\alpha:
 \alpha\in\Gamma\} \cup \{B_t: t\in T\}$
  is a special $(\Gamma\times[3])$-GBTD-starter.
 Then condition (vi) of \ref{defn:starter} ensures that the cell $(0,0)$ is special.
 
 On the other hand, if $\SSS$ be a $3$-$*$colorable $(\Gamma\times[3])$-GBTD-starter and let
 \begin{equation*}
 c_i \mbox{ be the color assigned to }
 \begin{cases}
 A_i-i,  & \mbox{if } i\in\Gamma,\\
 B_i,        & \mbox{otherwise.}
 \end{cases}
 \end{equation*}
 For $\alpha,\beta\in \Gamma$ and $t\in T$, assign the block $A_\alpha+\beta$ color $c_\alpha$ and the block $B_t+\beta$ color $c_t$.
 Then conditions (vii) and (viii) of Definition \ref{defn:starter} ensure that the ${\rm GBTD}_1(3,m)$ is $3$-$*$colorable with property $\Pi$.
 \end{proof}

\begin{proposition}\label{prop:fqstarterforGBTD(3,m)} Let $q\equiv 1\bmod 6$. Then there exists a special $(\FF_q\times[3])$-GBTD starter
that is $3$-$*$colorable with property $\Pi$.
\end{proposition}

 \begin{proof}
 Let $s=(q-1)/6$ and $\omega$ be a primitive element of $\FF_q$.
 Consider $\gamma\in\FF_q$ that satisfies the following conditions (note that $\omega^{2s}$ has order three):
 \begin{enumerate}[(A)]
 \item $\gamma\notin \{0,-1,-\omega^{2s},-\omega^{4s}\}$;
 \item $\gamma\notin \left\{\dfrac{\omega^{2is}-\omega^{t+2js}}{\omega^t-1}: i\ne j\in [3], t\in [s-1] \right\}$.
 \end{enumerate}
 The existence of $\gamma$ is guaranteed since the cardinality of the
 union of sets in (A) and (B) is at most $4+ 6(s-1)<6s+1=q$.

 Define $\Lambda$ to be $\left\{-\gamma \omega^{t-1+2(j-1)s} : t\in [s], j\in [3]\right\}$ and
construct the following $q+3s=(3q-1)/2$ blocks.
 For $\alpha\in \FF_q$, let
 {
 \begin{equation*}
 A_{\alpha}  =
 \begin{cases}
 \left\{\left(\omega^{t-1+2(j-1)s}\right)_i: j\in [3]\right\},           & \mbox{if $\alpha=-\gamma\omega^{t-1+2(i-1)s}$} 
 \mbox{ where $t\in [s]$, $i\in [3]$,}\\
 \left\{\left(-\frac{\alpha}{\gamma} \omega^{2(i-1)s}\right)_i:i\in
 [3])\right\},    & \mbox{otherwise.}
 \end{cases}
 \end{equation*}
 } For $(t,j)\in [s]\times [3]$, let
 \begin{equation*}
 B_{(t,j)}  =
 \left\{\left(\omega^{t-1+2(j-1)s}\left(\omega^{2(i-1)s}+\gamma\right)\right)_i:
 i \in [3]\right\}
 \end{equation*}

 Let $\SSS=\{A_\alpha: \alpha\in \FF_q\}\cup \{B_{(t,j)}: (t,j)\in [s]\times [3]\}$ and
 we claim that $\SSS$ is the desired starter.

Define
\begin{equation*}
\mathcal D=\{\{\omega^{t-1+2(j-1)s}: j\in [3]\}: t\in [s]\},
\end{equation*} 
\noindent and Wilson \cite{Wilson:1972d} showed that the blocks in $\mathcal D$ are mutually disjoint and
$\Delta \mathcal D=\FF_q\setminus\{0\}$.

Hence, for condition (i) of Definition \ref{defn:starter}, we check
for $i\in [3]$,
\begin{align*}
\Delta_{ii} \SSS
&= \Delta_{ii} \{A_\alpha: \alpha= -\gamma \omega^{t-1+2(i-1)s}, t\in[s], i \in [3]\} \\
&= \Delta \D = \FF_q\setminus\{0\}.
\end{align*}

For condition (ii), 
we verify for $i\ne i'\in [3]$, {
\begin{align*}
\Delta_{ii'}\SSS 
&= \bigcup_{\alpha\notin \Lambda} \left( -\frac{\alpha}{\gamma}
\left(\omega^{2(i-1)s}-\omega^{2(i'-1)s}\right)\right)
\cup \bigcup_{(t,j)\in [s]\times [3]} \omega^{t-1+2(j-1)s}\left(\omega^{2(i-1)s}-\omega^{2(i'-1)s}\right) \\
&= \left(\omega^{2(i-1)s}-\omega^{2(i'-1)s}\right)
 \left(\bigcup_{\alpha\notin\Lambda} -\frac{\alpha}{\gamma} \cup \bigcup_{(t,j)\in [s]\times [3]} \omega^{t-1+2(j-1)s} \right)\\
&= \left(\omega^{2(i-1)s}-\omega^{2(i'-1)s}\right) \FF_q=\FF_q.
\end{align*}
}

For condition (iii) of Definition \ref{defn:starter}, since the
number of points in $\bigcup_{\alpha\in \FF_q} A_\alpha$ is
$kq$, it suffices to check that each point $\beta_i\in\FF_q\times
[k]$ belongs to some block $A_\alpha$.
Indeed, if $\beta/\omega^{2(i-1)s}=\omega^{(t-1)+2(j-1)s}$ 
for some $(t,j)\in [s]\times[3]$, then let $\alpha=-\gamma\omega^{t-1+2(i-1)s}$ and so,
$\beta_i=\left(\omega^{t-1+2(i+j-2)s}\right)_i$ belongs to $A_\alpha$.
Otherwise,  $-\gamma\beta/\omega^{2(i-1)s}\notin\Lambda$. Let
$\alpha=-\gamma\beta/\omega^{2(i-1)s}$ and $\beta_i\in A_\alpha$ as desired.


Condition (iv) of Definition \ref{defn:starter} is clearly true from
the definition of $B_{(t,j)}$. We establish condition (v) of
Definition \ref{defn:starter} through the following claims:

\begin{claim}\label{claim:parallel}
The blocks in $\bigcup_{\alpha\notin\Lambda}
(A_\alpha-\alpha) \cup \bigcup_{(t,j)\in [s]\times [3]}
B_{(t,j)}$ form a resolution class.
\end{claim}

As above, it suffices to check that each point
$\beta_i\in\FF_q\times [3]$ belongs to some block in
$\bigcup_{\alpha\notin\Lambda} (A_\alpha-\alpha) \cup
\bigcup_{(t,j)\in [s]\times [k]} B_{(t,j)}$ as the total
number of points is $kq$.

Indeed, if $\beta/(\omega^{2(i-1)s}+\gamma)=\omega^{t-1+2(j-1)s}$ 
for some $(t,j)\in [s]\times [k]$, then $\beta_i\in B_{(t,j)}$. Otherwise,
$-\gamma\beta/(\omega^{2(i-1)s}+\gamma) \notin\Lambda$. 
Let $\alpha=-\gamma\beta/(\omega^{2(i-1)s}+\gamma)$ 
(note that $\alpha$ is well-defined by Condition (A)) and
$\beta_i\in A_\alpha-\alpha$.

\begin{claim}\label{claim:partial}
Each point in $\FF_q\times[k]$ appears at most once in
$\bigcup_{\alpha\in\Lambda}
\left(A_\alpha-\alpha\right)$.
\end{claim}

Note that the blocks are of the form
\begin{equation*}
\left\{\left(\omega^{t-1+2(j-1)s}+\gamma\omega^{t-1+2(i-1)s}\right)_i:j \in[3]\right\}
\end{equation*}
\noindent for $(t,i)\in [s]\times[3]$. Suppose otherwise that a
point appears twice. That is, there exist $j,j'\in [3]$,
$(t,i),(t',i)\in [s]\times [3]$ with $t>t'$ such that
\begin{equation*}
\omega^{t-1+2(j-1)s}+\gamma\omega^{t-1+2(i-1)s} 
=  \omega^{t'-1+2(j'-1)s}+\gamma\omega^{t'-1+2(i-1)s}.
\end{equation*}
Hence,
\begin{equation*}
\gamma=\frac{\omega^{2(j'-i)s}-\omega^{2(j-i)s+(t-t')}}{\omega^{t-t'}-1}.
\end{equation*}

Since $t\ne t'$, we have $t-t'\in [s-1]$. 
If $j\ne j'$, this contradicts Condition (B).
Otherwise $j=j'$ implies $\gamma=-\omega^{2(j-i)s}$ contradicting (A).

Next, observe that $A_0=\{(0,i): i\in [3]\}$. By
Claim \ref{claim:parallel}, to establish condition (vi) of Definition \ref{defn:starter}, 
it suffices to show
that $0_i\notin A_\alpha-\alpha$ for $\alpha\in\Lambda$
and $i\in[3]$. Suppose otherwise. Then there exists $(t,j)\in
[s] \times [3]$ and $i\in [3]$ such that
\begin{equation*}
(\omega^{(j-1)s}+\gamma)\omega^{t+(i-1)s}  = 0,
\end{equation*}
\noindent contradicting (A).

 Finally, we exhibit that $\SSS$ is $3$-$*$colorable with property $\Pi$ by assigning the block $A_0$ color $\clubsuit$,
 the blocks $A_\alpha-\alpha$ for $\alpha\notin\Lambda$ and $B_t$ for
 $t\in T$ color $\heartsuit$ and the blocks $A_\alpha-\alpha$ for
 $\alpha\in\Lambda$ color $\diamondsuit$. Then this assignment
 satisfies condition (vii) of Definition \ref{defn:starter}. In
 addition, $0_1$ is a witness for both $\heartsuit$ and
 $\diamondsuit$ and $\alpha_1$ is a witness for $\clubsuit$ for some
 $\alpha\ne 0$, satisfying condition (viii) of Definition
 \ref{defn:starter}.
 \end{proof}

\begin{corollary}\label{cor:fqgbtd} Let $q\equiv 1\bmod 6$. Then a $3$-$*$colorable ${\rm GBTD}_1(3,m)$ with property $\Pi$ exists.
\end{corollary}

\begin{proof}
This follows from Proposition \ref{prop:starter-gbtd} and
Proposition \ref{prop:fqstarterforGBTD(3,m)}.
\end{proof}

\begin{corollary}
\label{cor:smallgbtd}
A special ${\rm GBTD}_1(3,m)$ exists for
$m\in\{1,17,29,35,47,53,55\}$, a $3$-$*$colorable special ${\rm GBTD}_1(3,m)$ with property $\Pi$ for $m\in\{9,11,23\}$
and a $3$-$*$colorable
${\rm RBIBD}(15,3,1)$ with property $\Pi$.
\end{corollary}

\begin{proof}
A special ${\rm GBTD}_1(3,1)$ exists trivially.
In addition, a $3$-$*$colorable special ${\rm GBTD}_1(3,9)$ with property $\Pi$ is given by Example \ref{eg:9}, and
a $3$-$*$colorable $\rbibd(15,3,1)$ with property $\Pi$ is given by Example \ref{eg:5}.

For $m\in\{11,17,23,29,35,47,53,55\}$, apply Proposition \ref{prop:starter-gbtd} with special $(\ZZ_m\times[3])$-GBTD-starters
and $3$-$*$colorable special $(\ZZ_m\times[3])$-GBTD-starters with property $\Pi$
given in \cite{Cheeetal:2012online}.
\end{proof}
\subsection{Direct Constructions for an IGBTP$_1(\{2,3^*\},2m+w,(m+(w-1)/{2})\times (2m+w-4);w,({w-1})/{2}\times (w-4))$}

As with GBTDs, we use a set of starters to construct GBTPs. 
To construct this starters, we need the notion of {\em infinite elements}, 
or {\em intransitive starters}.

Given an abelian group $\Gamma$, we augment the point set 
with {\em infinite} elements, denoted by $\infty_i$ where $i$ belongs to some index set $I$.
The infinite elements are fixed under addition by elements in $\Gamma$. 
That is, $\infty_i+\gamma=\infty_i$ for $\gamma\in \Gamma$. 
Let $w$ be a positive integer and $W_w=\{\infty_i: i\in [w]\}$.
So, given a block $A\subset\Gamma\cup W_w$ and $\gamma\in \Gamma$, 
$A+\gamma=\{a+\gamma: a\in A\setminus W_w\}\cup (A\cap W_w)$.

We also extend the definition of difference lists. For a set system $(\Gamma\cup W_w, \SSS)$, 
then the difference list of $\SSS$ is given by the multiset
\begin{equation*}
\Delta\SSS=\langle x-y: x,y\in A\setminus W_w, x\ne y, A\in \SSS \rangle.
\end{equation*}

 \begin{figure}

\renewcommand{\arraystretch}{2}
\setlength{\tabcolsep}{5pt} \centering

\begin{tabular}{|c|c|c|}
\hline
$\vW$	& $\vB$  & $\vB+0_1$     \\\hline
{$\vA$}    & $\vC$  & $\vC+0_1$     \\ \hline
\end{tabular}
\vskip 5pt

\flushleft where $\vW$ is a $(w-1)/2\times (w-4)$ empty array, $\vA$ is an $m\times (w-4)$ array,
\vskip 5pt
\renewcommand{\arraystretch}{1.2}
\setlength{\tabcolsep}{1pt} 
\centering \small
\begin{tabular}{|c|ccccccc|}
\hline
$\{0_0,0_1\}$ 			& $A_{1}$ & $A_1+0_1$		& $A_{2}$ & $A_{2}+0_1$ & $\cdots$  			
& $A_{(w-5)/2}$ &  $A_{(w-5)/2}+0_1$\\
$\{1_0,1_1\}$ 			& $A_{1}+1_0$ & $A_1+1_1$ 	& $A_{2}+1_0$ & $A_{2}+1_1$ & $\cdots$  
& $A_{(w-5)/2}+1_0$ & $A_{(w-5)/2}+1_1$\\
$\vdots$ 				& $\vdots$ & $\vdots$			& $\vdots$ & $\vdots$ & $\ddots$ 		
& $\vdots$ & $\vdots$ \\
$\{(m-1)_0,(m-1)_1\}$ 	& $A_{1}-1_0$ & $A_{1}-1_1$ 	& $A_{2}-1_0$ & $A_{2}-1_1$ & $\cdots$   
 & $A_{(w-5)/2}-1_0$ & $A_{(w-5)/2}-1_1$\\
\hline
\end{tabular}\ ,

\vskip 5pt
\normalsize\flushleft $\vB$ and $\vC$ are the following $(w-1)/2\times m$ and $m\times m$ arrays,
\vskip 5pt
\renewcommand{\arraystretch}{1.2}
\setlength{\tabcolsep}{1pt} 
\centering \small
\begin{tabular}{c@{ , }c}
\begin{tabular}{|cccc|}
\hline
$B_1$ & $B_1+1_0$ & $\cdots$ & $B_1-1_0$ \\ 
$B_2$ & $B_1+1_0$ & $\cdots$ & $B_1-1_0$ \\ 
$\vdots$ & $\vdots$ & $\ddots$ & $\vdots$ \\
$B_{(w-1)/2}$ & $B_{(w-1)/2}+1_0$ & $\cdots$ & $B_{(w-1)/2}-1_0$ \\
\hline
\end{tabular}
&
\begin{tabular}{|cccc|}
\hline
$C_0$ 		& $C_{m-1}+1_0$ 	& $\cdots$ & $C_1-1_0$ \\ 
$C_1$ 		& $C_0+1_0$ 		& $\cdots$ & $C_2-1_0$ \\ 
$\vdots$ 		& $\vdots$ 		& $\ddots$ & $\vdots$ \\
$C_{m-1}$ 	& $C_{m-2}+1_0$ 	& $\cdots$ & $C_{0}-1_0$ \\
\hline
\end{tabular}
\end{tabular}\ .
 \caption{An ${\rm IGBTP}_1(\{2,3^*\},2m+w,(m+(w-1)/2)\times(2m+w-4);w,(w-1)/2\times (w-4))$ 
 from a $((\ZZ_m\times \ZZ_2) \cup W_w)$-GBTP-starter.}
 \label{fig:GBTPstarter-Z2}
\end{figure}

 \begin{figure}

\renewcommand{\arraystretch}{2}
\setlength{\tabcolsep}{5pt} \centering

\begin{tabular}{|c|c|c|c|c|}
\hline
$\vW$				    & $\vB$  & $\vB+0_1$   & $\vB+0_2$ & $\vB+0_3$  \\\hline
\multirow{2}{*}{$\vA$}         & $\vC$  & $\vD+0_1$  &  $\vC+0_2$ & $\vD+0_3$  \\\cline{2-5}
                                        & $\vD$  & $\vC+0_1$  &  $\vD+0_2$ & $\vC+0_3$  \\ \hline
\end{tabular}
\vskip 5pt
\small
\flushleft where $\vW$ is a $4\times 5$ empty array, $\vA$ is a $2m\times 5$ array,
\vskip 5pt
\renewcommand{\arraystretch}{1.2}
\setlength{\tabcolsep}{2pt} 
\centering \small
\begin{tabular}{|ccc|cc|}
\hline
$\{0_0,0_1\}$ & $\{x_0,x_2\}$ & $\{y_0,y_3\}$ 
& $A$ & $A+0_2$ \\
$\{1_0,1_1\}$ & $\{(x+1)_0,x_2\}$ & $\{(y+1)_0,(y+1)_3\}$ 
& $A+1_0$ & $A+1_2$ \\
$\vdots$ & $\vdots$ & $\vdots$ & $\vdots$ & $\vdots$ \\
$\{(m-1)_0,(m-1)_1\}$ & $\{(x-1)_0,x_2\}$ & $\{(y-1)_0,(y-1)_3\}$ 
& $A+(m-1)_0$ & $A+(m-1)_2$\\
\hline
$\{0_2,0_3\}$ & $\{x_1,x_3\}$ & $\{y_1,y_2\}$ 
& $A+0_1$ & $A+0_3$\\
$\{1_2,1_3\}$ & $\{(x+1)_1,x_3\}$ & $\{(y+1)_1,(y+1)_2\}$ 
& $A+1_1$ & $A+1_3$\\
$\vdots$ & $\vdots$ & $\vdots$ & $\vdots$ & $\vdots$ \\
$\{(m-1)_2,(m-1)_3\}$ & $\{(x-1)_1,x_3\}$ & $\{(y-1)_1,(y-1)_2\}$ 
& $A+(m-1)_1$ & $A+(m-1)_3$\\
\hline
\end{tabular}\ ,

\vskip 5pt
\flushleft\normalsize $\vB$, $\vC$ and $\vD$ are the following $4\times m$, $m\times m$ and $m\times m$ arrays respectively,
\vskip 5pt
\renewcommand{\arraystretch}{1.2}
\setlength{\tabcolsep}{1pt} 
\centering \small
\begin{tabular}{c@{ , }c@{ , }c}
\begin{tabular}{|cccc|}
\hline
$B_1$ & $B_1+1_0$ & $\cdots$ & $B_1-1_0$ \\ 
$B_2$ & $B_2+1_0$ & $\cdots$ & $B_2-1_0$ \\ 
$B_3$ & $B_3+1_0$ & $\cdots$ & $B_3-1_0$ \\
$B_4$ & $B_4+1_0$ & $\cdots$ & $B_4-1_0$ \\  

\hline
\end{tabular}
&
\begin{tabular}{|cccc|}
\hline
$C_0$ 		& $C_{m-1}+1_0$ 	& $\cdots$ & $C_1-1_0$ \\ 
$C_1$ 		& $C_0+1_0$ 		& $\cdots$ & $C_2-1_0$ \\ 
$\vdots$ 		& $\vdots$ 		& $\ddots$ & $\vdots$ \\
$C_{m-1}$ 	& $C_{m-2}+1_0$ 	& $\cdots$ & $C_{0}-1_0$ \\
\hline
\end{tabular}
&
\begin{tabular}{|cccc|}
\hline
$D_0$ 		& $D_{m-1}+1_0$ 	& $\cdots$ & $D_1-1_0$ \\ 
$D_1$ 		& $D_0+1_0$ 		& $\cdots$ & $D_2-1_0$ \\ 
$\vdots$ 		& $\vdots$ 		& $\ddots$ & $\vdots$ \\
$D_{m-1}$ 	& $D_{m-2}+1_0$ 	& $\cdots$ & $D_{0}-1_0$ \\
\hline
\end{tabular} 
\end{tabular}\ .
 
 \caption{An ${\rm IGBTP}_1(\{2,3^*\},4m+9,(2m+4)\times(4m+5);9,4\times 5)$ from a $((\ZZ_m\times \ZZ_4) \cup W_9)$-GBTP-starter.}
 \label{fig:GBTPstarter-Z4}
\end{figure} 

 \begin{definition}
 \label{defn:starter-igbtp-Z2} Let $m$ be an odd integer with $m\geq 11$ 
 Let $(\ZZ_m\times\ZZ_2\cup W_w,\SSS)$ be a $\{2,3\}$-uniform set system of size $w-3+m$, where
 \begin{equation*}
 \SSS =\{A_i: i\in [(w-5)/2]\} \cup \{B_i: i\in [(w-1)/2]\}\cup \{C_i: i\in \ZZ_{m}\}.
 \end{equation*}
 satisfying $|A_i|=2$ for $i\in [(w-5)/2]$, $|B_i|=2$ for $i\in [(w-1)/2]$,
 $|C_0|=3$, and $|C_i|=2$ for $i\in \ZZ_m\setminus\{0\}$.

$\SSS$ is called a {\em $((\ZZ_m\times\ZZ_2)\cup W_w)$-IGBTP-starter} if the following
conditions hold:
\begin{enumerate}[(i)]
\item $\Delta\SSS=\ZZ_m\times\ZZ_2\setminus\{0_0, 0_1\}$,
\item $\{j: a_j\in A_i\}=\ZZ_2$ for $i\in [(w-5)/2]$,
\item $\{B_i: i\in [(w-1)/2]\}\cup \{C_j: j\in \ZZ_{m}\}=(\ZZ_m\times\ZZ_2)\cup W_w$,
\item $|C_i\cap W_w|\le 1$ for $i\in \ZZ_m$,
\item each element in $(\ZZ_{m}\times\ZZ_2) \cup W_w$ appears either once or twice in the multiset
\begin{equation*}
R =  \{0_0,0_1\} \cup \left( \bigcup_{\substack{i\in[(w-5)/2]\\ j\in\ZZ_2}} A_i+0_j \right)
 \cup \left( \bigcup_{i_j\in \ZZ_{m}\times \ZZ_2} C_i-i_j\right).
\end{equation*}
\end{enumerate}

\end{definition}

 \begin{proposition}\label{prop:starter-igbtp-Z2}
 Suppose there exists a $((\ZZ_m\times \ZZ_2)\cup W_w)$-IGBTP-starter.
 Then there exists
 an IGBTP$_1(\{2,3^*\},2m+w,(m+(w-1)/2)\times (2m+w-4);w,(w-1)/2\times (w-4) )$.
 \end{proposition}

\begin{proof}
Let
\begin{align*}
X &= \ZZ_{m}\times\ZZ_2\cup W_w, \\
\A &= \{S+j:\text{$S\in \SSS$ and $j\in\ZZ_{m}\times\ZZ_2$}\}\cup\{\{i_0,i_1\}: i\in\ZZ_m\}.
\end{align*}
Then $(X,W_w,\A)$ is an IRP$(2m+w,K,1;w)$, whose blocks can be arranged in an
$(m+(w-1)/2)\times (2m+w-4)$ array as in Figure \ref{fig:GBTPstarter-Z4}.
We index the rows by $[(w-1)/2]\cup \ZZ_m$ and the columns by $[w-4]\cup(\ZZ_m\times \ZZ_2)$.

First, check that the cell $(r,c)$ is empty for $(r,c)\in [(w-1)/2]\times [w-4]$.

For $j\in[w-4]$, the set of blocks occupying column $j$ is $\ZZ_m\times\ZZ_2$
by condition (ii) of Definition \ref{defn:starter-igbtp-Z2}.
For $j\in \ZZ_m\times \ZZ_2$, first observe that 
the set of the blocks occupying the column $0_0$ by condition (iii) of Definition \ref{defn:starter-igbtp-Z2} 
is $(\ZZ_m\times \ZZ_2)\cup W_w$. 
Since the blocks of column $j$ are translates (by $j$) of the blocks in column $0_0$, 
the union of the blocks in column $j$ is also $(\ZZ_m\times \ZZ_2)\cup W_w$.

For $i\in[(w-1)/2]$, each element in $\ZZ_m\times \ZZ_2$ appears exactly twice in row $i$ by construction.
For $i\in\ZZ_m$, let $R_{i}$ denote the multiset containing all the points
appearing in the blocks of row $i$. Then $R_0=R$ and $R_{i} = R_{0}+i_0$,
for all $i\in\ZZ_{m}$. Hence, it suffices each element in $X$ appears either once or twice in $R$,
which follows immediately from conditions (v) in Definition \ref{defn:starter-igbtp-Z2}.
\end{proof}

 \begin{definition}
 \label{defn:starter-igbtp-Z4} Let $m$ be an odd integer with $m\geq 11$. 
 Let $((\ZZ_m\times\ZZ_4)\cup W_9,\SSS)$ be a $\{1,2,3\}$-uniform set system of size $7+2m$, where
 \begin{equation*}
 \SSS =\{x_0\}\cup \{y_0\} \cup A\cup \{B_i: i\in [4]\}\cup \{C_i: i\in \ZZ_{m}\} \cup \{D_i: i\in \ZZ_{m}\}.
 \end{equation*}
 satisfying $|A|=2$, $|B_i|=2$ for $i\in [4]$,
 $|C_0|=3$, $|C_i|=2$ for $i\in \ZZ_m\setminus\{0\}$ and $|D_i|=2$ for $i\in\ZZ_m$.

$\SSS$ is called a {\em $((\ZZ_m\times\ZZ_4)\cup W_9)$-IGBTP-starter} if the following
conditions hold:
\begin{enumerate}[(i)]
\item $\Delta\SSS=(\ZZ_m\times\ZZ_4)\setminus\{0_0, 0_1,0_2,0_3\}$,
\item $\{j: a_j\in A\}=\{0,2\}$,
\item $\{B_i: i\in [(w-1)/2]\}\cup \{C_i: i\in \ZZ_{m}\}\cup \{D_i: i\in \ZZ_{m}\}=(\ZZ_m\times\ZZ_4)\cup W_9$,
\item $|C_i\cap W_9|\le 1$ and $|D_i\cap W_9|\le 1$ for $i\in\ZZ_m$,
\item each element in $(\ZZ_{m}\times\ZZ_4) \cup W_9$ appears either once or twice in the multisets

{\small
\begin{align*}
R_\circ &=  \{0_0,0_1,x_0,x_2,y_0,y_3\} \cup A\cup A+0_2
 \cup \left( \bigcup_{i\in \ZZ_{m}, j\in \{0,2\}} C_i-i_j\right)
 \cup \left( \bigcup_{i\in \ZZ_{m}, j\in\{1,3\}} D_i-i_j\right),\\
 R_\bullet &= \{0_2,0_3,x_1,x_3,y_1,y_2\} \cup A+0_1\cup A+0_3
  \cup \left( \bigcup_{i\in \ZZ_{m}, j\in \{1,3\}} C_i-i_j\right)
 \cup \left( \bigcup_{i\in \ZZ_{m}, j\in\{0,2\}} D_i-i_j\right).
\end{align*}}
\end{enumerate}
\end{definition}

 \begin{proposition}\label{prop:starter-igbtp-Z4}
 Suppose there exists a $(\ZZ_m\times \ZZ_4\cup W_9)$-IGBTP-starter.
 Then there exists
 an IGBTP$_1(\{2,3^*\},4m+9,(2m+4)\times (4m+5);9,4\times 5 )$.
 \end{proposition}

\begin{proof}
Let
\begin{align*}
X &= (\ZZ_{m}\times\ZZ_4)\cup W_9, \\
\A &= \{S+j:\text{$S\in \SSS$, $|S|\ne 1$, $j\in\ZZ_{m}\times\ZZ_2$}\}
\cup\{\{i_0,i_1\}: i\in\ZZ_m\}\cup\{\{i_2,i_3\}: i\in\ZZ_m\}\\
& ~~~ \cup\{\{(x+i)_0,(x+i)_2\}: i\in\ZZ_m\} \cup \{\{(x+i)_1,(x+i)_3\}: i\in\ZZ_m\}\\
& ~~~\cup \{\{(y+i)_0,(y+i)_3\}: i\in\ZZ_m\} \cup \{\{(y+i)_1,(y+i)_2\}: i\in\ZZ_m\}.
\end{align*}
Then $(X,W_9,\A)$ is an IRP$(4m+9,K,1;9)$, whose blocks can be arranged in a
$(2m+4)\times (4m+5)$ array as in Figure \ref{fig:GBTPstarter-Z2}.
We index the rows by $[4]\cup (\ZZ_m\times\{\circ,\bullet\})$ and the columns by $[5]\cup(\ZZ_m\times \ZZ_4)$.

First, check that the cell $(r,c)$ is empty for $(r,c)\in [4]\times [5]$.

For $j\in[5]$, the set of blocks occupying column $j$ is $\ZZ_m\times\ZZ_4$
by condition (ii) of Definition \ref{defn:starter-igbtp-Z4}.
For $j\in \ZZ_m\times \ZZ_4$, first observe that 
the set of the blocks occupying the column $0_0$ by condition (iii) of Definition \ref{defn:starter-igbtp-Z4} 
is $(\ZZ_m\times \ZZ_4)\cup W_9$. 
Since the blocks of column $j$ are translates (by $j$) of the blocks in column $0_0$, 
the union of the blocks in column $j$ is also $(\ZZ_m\times \ZZ_4)\cup W_9$.

For $i\in[4]$, each element in $\ZZ_m\times \ZZ_4$ appears exactly twice in row $i$ by construction.
For $(i,*)\in\ZZ_m\times\{\circ,\bullet\}$, let $R_{(i,*)}$ denote the multiset containing all the points
appearing in the blocks of row $(i,*)$. Then  $R_{(0,*)}=R_*$ and $R_{(i,*)} = R_{(0,*)}+i_0$,
for all $i\in\ZZ_{m}$. 
Hence, it suffices each element in $X$ appears either once or twice in $R_*$,
which follows immediately from conditions (v) in Definition \ref{defn:starter-igbtp-Z4}.
\end{proof}

 \begin{corollary}
 \label{cor:small-igbtp}
An IGBTP$_1(\{2,3^*\},2m+9,(m+4)\times (2m+5); 9,4\times 5)$ exists
 for $m\in \{s: 10\le s\le 45\}\cup\{47,49,53,57,77\}\setminus\{16,20,24,28,36,40,44\}$, 
and an IGBTP$_1(\{2,3^*\},2m+11,(m+5)\times (2m+7); 11,5\times 7)$ exists
 for $m\in \{15,19,23,27,31$, $35,45,49\} $.

 \end{corollary}
 \begin{proof}
The required $((\ZZ_m\times \ZZ_2)\cup W_9)$-IGBTP-starter for $m\in \{s: 11\leq s\leq 49, s\mbox{ odd}\} \cup \{53,57,77\}$
and $((\ZZ_m\times \ZZ_4)\cup W_9)$-IGBTP starter for $m\in \{s:5\le s\le 21, s\mbox{ odd}\}$ is given in \cite{Cheeetal:2012online} and 
 we apply Proposition \ref{prop:starter-igbtp-Z2} and Proposition \ref{prop:starter-igbtp-Z4} to obtain the corresponding IGBTP.
 
 Similarly, to construct an IGBTP$_1(\{2,3^*\},2m+11,(m+5)\times (2m+7); 11,5\times 7)$ for $m\in \{15,19,23,27,31,35,45,49\} $, 
 we apply Proposition \ref{prop:starter-igbtp-Z2} to $(\ZZ_m\times \ZZ_2\cup W_{11})$-IGBTP starters listed in \cite{Cheeetal:2012online}.
 
 It remains to construct an IGBTP$_1(\{2,3^*\},33,16\times 29; 9,4\times 5)$. 
 Consider $((\ZZ_{3}\times\ZZ_8)\cup W_9,\SSS)$, a $\{2,3\}$-uniform set system of size $36$, 
 where
 $\SSS$ comprise the blocks below: 

\begin{center}
\renewcommand{\arraystretch}{1.1}
{\setlength{\tabcolsep}{3pt}
\begin{tabular}{llll}
$A_1=\{1_0,1_2\}$      &
$A_2=\{1_1,1_5\}$      &
$A_3=\{0_0,0_4\}$      &
$A_4=\{1_3,1_6\}$      \\
$A_5=\{0_3,0_5\}$      &
$A_6=\{1_1,1_3\}$      &
$A_7=\{1_4,1_7\}$      &
$A_8=\{0_1,0_6\}$      \\
$A_9=\{0_0,0_5\}$      &
$A_{10}=\{0_2,0_4\}$      &
$A_{11}=\{1_4,1_6\}$      &
$A_{12}=\{1_0,1_3\}$      \\
$A_{13}=\{0_2,0_5\}$      &
$A_{14}=\{1_2,1_7\}$      &
$A_{15}=\{0_1,0_7\}$      &
$A_{16}=\{1_5,1_7\}$      \\
$A_{17}=\{0_2,0_6\}$      &
$A_{18}=\{0_3,0_7\}$      &
$A_{19}=\{1_1,1_4\}$      &
$A_{20}=\{1_0,1_6\}$      \\
$B_1=\{0_0,0_1\}$      &
$B_2=\{0_5,1_5\}$      &
$B_3=\{1_1,2_4\}$      &
$B_4=\{0_7,1_3\}$      \\
$C^{1}_0=\{1_0,2_1, 2_6\}$      &
$C^{1}_1=\{1_0,2_1\}$      &
$C^{1}_2=\{1_0,2_1\}$      \\
$C^{2}_0=\{0_2,\infty_1\}$      &
$C^{2}_1=\{0_4,\infty_2\}$      &
$C^{2}_2=\{1_2,\infty_3\}$      \\
$C^{3}_0=\{2_0,\infty_4\}$      &
$C^{3}_1=\{2_3,\infty_5\}$      &
$C^{3}_2=\{1_6,\infty_6\}$      \\
$C^{4}_0=\{2_7,\infty_7\}$      &
$C^{4}_1=\{2_2,\infty_8\}$      &
$C^{4}_2=\{2_5,\infty_9\}$.      \\
\end{tabular}}
\end{center}

Let \begin{align*}
 X &= (\ZZ_3\times \ZZ_8) \cup W\\
 \A &=\{S+j: S\in\SSS, j\in\ZZ_3\times\ZZ_8\}.
 \end{align*}
 Then $(X,W,\A)$ is an IRP$(33,\{2,3^*\},1;9)$, whose blocks can be arranged in a $16\times 29$ array as in Figure \ref{fig:IGBTP-24-9}.
 It can be readily verified that this arrangement results in an IGBTP$_1(\{2,3^*\},33,16\times 29; 9,4\times 5)$.
  \end{proof}
 
\begin{figure}
\renewcommand{\arraystretch}{2}
\setlength{\tabcolsep}{3pt} \centering

\begin{tabular}{|c| c|c|c|c| c|c|c|c|}
\hline
$\vW$				    & $\vB$  & $\vB+0_1$   & $\vB+0_2$ & $\vB+0_3$ & $\vB+0_4$  & $\vB+0_5$   & $\vB+0_6$ & $\vB+0_7$  \\\hline
\multirow{4}{*}{$\vA$} 
& $\vC_1$  & $\vC_4+0_1$  &  $\vC_3+0_2$ & $\vC_2+0_3$ & $\vC_1+0_4$  & $\vC_4+0_5$  &  $\vC_3+0_6$ & $\vC_2+0_7$  \\\cline{2-9}
& $\vC_2$  & $\vC_1+0_1$  &  $\vC_4+0_2$ & $\vC_3+0_3$ & $\vC_2+0_4$  & $\vC_1+0_5$  &  $\vC_4+0_6$ & $\vC_3+0_7$  \\\cline{2-9}
& $\vC_3$  & $\vC_2+0_1$  &  $\vC_1+0_2$ & $\vC_4+0_3$ & $\vC_3+0_4$  & $\vC_2+0_5$  &  $\vC_1+0_6$ & $\vC_4+0_7$  \\\cline{2-9}
& $\vC_4$  & $\vC_3+0_1$  &  $\vC_2+0_2$ & $\vC_1+0_3$ & $\vC_4+0_4$  & $\vC_3+0_5$  &  $\vC_2+0_6$ & $\vC_1+0_7$  \\\hline
\end{tabular}

\vskip 5pt

\flushleft where $\vW$ is a $4\times 5$ empty array, $\vA$ is a $12\times 5$ array,
\vskip 5pt

\renewcommand{\arraystretch}{1.2}
\setlength{\tabcolsep}{2pt} 
\centering 
{\small
\begin{tabular}{|ccccc|}
\hline
$A_1$ & $A_2$ & $A_3$ & $A_4$ & $A_5$ \\
$A_1+1_0$ & $A_2+1_0$ & $A_3+1_0$ & $A_4+1_0$ & $A_5+1_0$ \\
$A_1+2_0$ & $A_2+2_0$ & $A_3+2_0$ & $A_4+2_0$ & $A_5+2_0$ \\

$A_6$ & $A_7$ & $A_8$ & $A_9$ & $A_{10}$ \\
$A_6+1_0$ & $A_7+1_0$ & $A_8+1_0$ & $A_9+1_0$ & $A_{10}+1_0$ \\
$A_6+2_0$ & $A_7+2_0$ & $A_8+2_0$ & $A_9+2_0$ & $A_{10}+2_0$ \\

$A_{11}$ & $A_{12}$ & $A_{13}$ & $A_{14}$ & $A_{15}$ \\
$A_{11}+1_0$ & $A_{12}+1_0$ & $A_{13}+1_0$ & $A_{14}+1_0$ & $A_{15}+1_0$ \\
$A_{11}+2_0$ & $A_{12}+2_0$ & $A_{13}+2_0$ & $A_{14}+2_0$ & $A_{15}+2_0$ \\

$A_{16}$ & $A_{17}$ & $A_{18}$ & $A_{19}$ & $A_{20}$ \\
$A_{16}+1_0$ & $A_{17}+1_0$ & $A_{18}+1_0$ & $A_{19}+1_0$ & $A_{20}+1_0$ \\
$A_{16}+2_0$ & $A_{17}+2_0$ & $A_{18}+2_0$ & $A_{19}+2_0$ & $A_{20}+2_0$ \\
\hline
\end{tabular}\ ,}

\vskip 5pt
\flushleft \normalsize $\vB$ is a $4\times 3$ array,
\vskip 5pt
\renewcommand{\arraystretch}{1.2}
\setlength{\tabcolsep}{1pt} 
\centering 
\begin{tabular}{|ccc|}
\hline
$B_1$ & $B_1+1_0$ & $B_1+2_0$ \\ 
$B_2$ & $B_2+1_0$ & $B_2+2_0$ \\ 
$B_3$ & $B_3+1_0$ & $B_3+2_0$ \\
$B_4$ & $B_4+1_0$ & $B_4+2_0$ \\  

\hline
\end{tabular}\ ,

\vskip 5pt
\flushleft $\vC_i$ for $i\in[4]$ is a $3\times 3$ array,
\vskip 5pt
\renewcommand{\arraystretch}{1.2}
\setlength{\tabcolsep}{1pt} 
\centering 

\begin{tabular}{|ccc|}
\hline
$C^i_0$ 	& $C^i_{2}+1_0$ & $C^i_1+2_0$ \\ 
$C^i_1$ 	& $C^i_0+1_0$ 	 & $C^i_2+2_0$ \\ 
$C^i_2$ 	& $C^i_{1}+1_0$ & $C^i_{0}+2_0$ \\
\hline
\end{tabular}\ .

 \caption{An ${\rm IGBTP}_1(\{2,3^*\},33,16\times29;9,4\times 5)$.}
 \label{fig:IGBTP-24-9}

\end{figure}

 \subsection{Direct Constructions for FrGBTDs}

 \begin{lemma}
 \label{lem:small-fgbtd-2} There exists an FrGBTD$(2,2^t)$ for
 $t\in \{4,5\}$.
 \end{lemma}
 \begin{proof}
 The desired FrGBTDs are given in Fig. \ref{fig:fgbtd2^4} and Fig. \ref{fig:fgbtd2^5}.
 \end{proof}
 \begin{figure}[t!]
 \renewcommand{\arraystretch}{1}
 \small \centering
 \begin{tabular}{|@{ }c@{ }|@{ }c@{ }|@{ }c@{ }|@{ }c@{ }|@{ }c@{ }|@{ }c@{ }|@{ }c@{ }|@{ }c@{ }|@{ }c@{ }|@{ }c@{ }|
 @{ }c@{ }|@{ }c@{ }|}\hline
 ---      &---     &\{2,7\}&\{6,3\}  &\{7,1\} &\{3,5\} &\{5,6\} &\{1,2\}\\ \hline
 \{2,3\} &\{6,7\} &--- &---           &\{3,0\} &\{7,4\} &\{0,2\} &\{4,6\}\\ \hline
 \{5,7\} &\{1,3\} &\{3,4\} &\{7,0\} &--- &---           &\{4,1\} &\{0,5\}\\ \hline
 \{1,6\} &\{5,2\} &\{6,0\} &\{2,4\} &\{4,5\} &\{0,1\} &--- &---           \\ \hline
 \end{tabular}
 \caption{An ${\rm FrGBTD}_1(2,2^4)$ $(X,\G,\A)$, where $X=\ZZ_{8}$
 and $\G= \{ \{i,4+i\}: i\in\ZZ_4\}$. } \label{fig:fgbtd2^4}
 \end{figure}

 \begin{figure}[t!]
 \renewcommand{\arraystretch}{1}
 \small \centering
 \begin{tabular}{|@{ }c@{ }|@{ }c@{ }|@{ }c@{ }|@{ }c@{ }|@{ }c@{ }|@{ }c@{ }|@{ }c@{ }|@{ }c@{ }|@{ }c@{ }|@{ }c@{ }|
 @{ }c@{ }|@{ }c@{ }|}\hline
 ---      &---    &\{7,9\}&\{2,4\}  &\{3,4\} &\{8,9\} &\{6,2\} &\{1,7\}&\{1,8\} &\{6,3\}\\ \hline
 \{7,4\} &\{2,9\} &---      &---    &\{8,0\} &\{3,5\} &\{4,5\} &\{9,0\}&\{7,3\} &\{2,8\}\\ \hline
 \{3,9\} &\{8,4\} &\{8,5\} &\{3,0\} &---      &---    &\{9,1\} &\{4,6\}&\{5,6\} &\{0,1\}\\ \hline
 \{1,2\} &\{6,7\} &\{4,0\} &\{9,5\} &\{9,6\} &\{4,1\} &---      &---   &\{0,2\} &\{5,7\}   \\ \hline
 \{6,8\} &\{1,3\} &\{2,3\} &\{7,8\} &\{5,1\} &\{0,6\} &\{0,7\} &\{5,2\}&---      &---      \\ \hline
 \end{tabular}
 \caption{An ${\rm FrGBTD}_1(2,2^5)$ $(X,\G,\A)$, where $X=\ZZ_{10}$
 and $\G= \{ \{i,5+i\}: i\in\ZZ_5\}$. } \label{fig:fgbtd2^5}
 \end{figure}

\begin{definition}\label{defn:fgbtdstarter}
Let $t$ be a positive integer, and let $I=[t-1]\times[2]$.
Let $(\ZZ_{3t}\times[2],\SSS)$ be a $3$-uniform set system
of size $2(t-1)$, where $\SSS=\{A_i:i\in I\}$. $\SSS$ is called a
{\em $(\ZZ_{3t}\times[2])$-FrGBTD-starter} if the following conditions hold:
\begin{enumerate}[(i)]
\item $\Delta_{ij}\SSS=\ZZ_{3t}\setminus\{0,t,2t\}$ for $i,j\in [2]$,
\item $\cup_{i\in I} A_i = (\ZZ_{3t}\setminus\{0,t,2t\})\times[2]$,
\item for $j\in[2]$, each element in $(\ZZ_t\setminus\{0\})\times[2]$ appears either once or twice in the multiset
\begin{equation*}
R_j = \bigcup_{i=1}^{t-1} A_{(i,j)}-i\bmod t,
\end{equation*}
\item $r\in (\ZZ_t\setminus\{0\})\times[2]$ for each $r\in R_1\cup R_2$.
\end{enumerate}
\end{definition}

\begin{proposition}\label{prop:fgbtdstarter}
If a $(\ZZ_{3t}\times[2],6^t)$-FrGBTD-starter exists, 
then an ${\rm FrGBTD}(3,6^t)$ exists.
\end{proposition}

\begin{proof}
Let
\begin{align*}
X &= \ZZ_{3t}\times[2], \\
\G &= \{ G_i=\{i_1,(t+i)_1,(2t+i)_1, i_2,(t+i)_2,(2t+i)_2\} : i\in\ZZ_t \},  \\
\A &= \{A_i+j:\text{$i\in I$ and $j\in\ZZ_{3t}$}\}.
\end{align*}
Then $(X,\G,\A)$ is a $\{3\}$-GDD of type $6^t$, whose blocks can be arranged in a
$2t\times 3t$ array, with rows and columns indexed by $\ZZ_t\times[2]$
and $\ZZ_{3t}$, respectively, as follows: the block $A_{(i,j)}+k$ is placed in
cell $((i+k,j),k)$.

The set of blocks occupying column zero are $\{A_i:i\in I\}$ and
by condition (ii) of Definition \ref{defn:fgbtdstarter}, $\bigcup_{i\in I} A_i=X\setminus G_0$.
For other $j\in\ZZ_{3t}$, observe that the blocks occupying column $j$ are translates
(by $j$) of the blocks in column zero, and hence the union of the blocks
in column $j$ is $X\setminus G_{j'}$, where $j'\equiv j\bmod t$.

For $(i,j)\in\ZZ_t\times [2]$, let $R_{(i,j)}$ denote the multiset containing all the points
appearing in the blocks of row $(i,j)$. Then $R_{(i,j)} = R_{(0,j)}+i$,
for all $i\in\ZZ_{t}$. Hence, it suffices to check that each element of $X \setminus G_0$
appears either once or twice in $R_{(0,j)}$ and the elements of $R_{(0,j)}$
belong to $X\setminus G_0$ for $j\in [2]$. This, however, follows immediately
from conditions (iii) and (iv) in Definition \ref{defn:fgbtdstarter}, since
$R_{(0,j)}=R_j\cup (R_j+t)\cup (R_j+2t)$ for $j\in [2]$.
\end{proof}

\begin{corollary}
\label{cor:smallfgbtd}
There exist an ${\rm FrGBTD}(3,6^t)$ for all
$t\in\{5,6,7,8\}$, an ${\rm FrGBTD}(3,24^t)$ for all $t\in\{5,8\}$
and an ${\rm FrGBTD}(3,30^t)$ for all $t\in\{5,7\}$.
\end{corollary}

\begin{proof}
An ${\rm FrGBTD}_1(3,6^6)$ is given by Example \ref{eg:fgbtd6}.
An ${\rm FrGBTD}(3,6^t)$ for $t\in\{5,7\}$ exists by
applying Proposition \ref{prop:fgbtdstarter} with FrGBTD-starters given in \cite{Cheeetal:2012online}.

The existence of an ${\rm FrGBTD}(3,24^t)$, $t\in\{5,8\}$ follows by applying
Proposition \ref{prop:inflation} with an ${\rm FrGBTD}(3,6^t)$ (constructed in this proof)
and a ${\rm DRTD}(3,4)$, whose existence is provided by Corollary \ref{drtdexistence}.
The existence of an ${\rm FrGBTD}(3,30^t)$, $t\in\{5,7\}$ follows by
applying Proposition \ref{prop:inflation} similarly.

To prove the existence of an ${\rm FrGBTD}(3,6^8)$, consider $(\ZZ_{48},\SSS)$, a $\{3\}$-uniform set system of size $7$, where $\SSS$ comprise
the blocks below:
\begin{align*}
A_{1}&=\{2,3,5\}&  
A_{2}&=\{4,14,31\} &  
A_{3}&=\{9,22,45\}&
A_{4}&=\{15,34,43\} \\  
A_{5}&=\{20,35,42\} &  
A_{6}&=\{13,17,47\}&
A_{7}&=\{1,6,12\}.
\end{align*}

\noindent Observe that $\SSS$ satisfies the following conditions:
\begin{enumerate}[(i)]
\item $\Delta\SSS=\ZZ_{48}\setminus\{0,8,16,24,32,40\}$,
\item $\cup_{i\in [7]} A_i\bmod 24 = \ZZ_{24}\setminus\{0,8,16\}$,
\item each element in $\ZZ_{16}\setminus\{0,8\}$ appears either once or twice in the multiset
\begin{equation*}
R = \bigcup_{i\in [7]} A_i-i\bmod {16},
\end{equation*}
\item $r\in \ZZ_{16}\setminus\{0,8\}$ for each $r\in R$.
\end{enumerate}

 Further, let \begin{align*}
 X &= \ZZ_{48},\\
 \G &= \{ \{i+8k: k\in\ZZ_6\} : i\in\ZZ_8\}, \\
 \A &=\{A_i+j: \text{$i\in [7]$ and $j\in\ZZ_{48}$}\}.
 \end{align*}
 Then $(X,\G,\A)$ is a $\{3\}$-GDD of type $6^8$, whose blocks can be arranged in a $16\times 24$ array,
 with rows and columns are indexed by $\ZZ_{16}$ and $\ZZ_{24}$, respectively, as follows: the block $A_i+j$ is placed in cell $(i+j,j)$.
 This array can be verified to be an ${\rm FrGBTD}(3,6^8)$.
 \end{proof}

 \section{Existence of GBTDs and GBTPs}

We apply recursive constructions in Section
 \ref{sec:recursive} with small designs directly constructed in
 Section \ref{sec:direct} to completely settle the existence of
 GBTD$_1(3,m)$ and GBTP$_1(\{2,3^*\};2m+1,m\times(2m-3))$.

 \subsection{Existence of GBTD$_1(3,m)$}

 \begin{lemma}
 \label{lem:3^rq} There exists a special ${\rm GBTD}_1(3,3^rq)$ for
 all $r\geq 0$ and $q\in Q$, where $Q=\{q:\text{$q\equiv 1\bmod 6$ is a
 prime power}\}\cup\{5,9,11,23\}$, except when $(r,q)=(0,5)$.
 \end{lemma}

\begin{proof}
Existence of a special ${\rm GBTD}_1(3,q)$ for all $q\in Q\setminus \{5\}$ is provided by
Corollary \ref{cor:fqgbtd} and \ref{cor:smallgbtd}. These
GBTDs are all $3$-$*$colorable with property $\Pi$. The lemma then follows by considering these
GBTDs as RBIBDs and applying Corollary \ref{cor:3^km}.
\end{proof}
 
\begin{lemma}
\label{lem:10n+1}Let $s\in [2]$ and suppose there exists a ${\rm
TD}(5+s,n)$. If $0\leq g_i\leq n$, $i\in[s]$ and that there exists a
special ${\rm GBTD}_1(3,m)$ for all
$m\in\{2n+1\}\cup\{2g_i+1:i\in[s]\}$, then there exists a special
${\rm GBTD}_1(3,10n+1+2\sum_{i=1}^s g_i)$.
\end{lemma}

\begin{proof}
By Corollary \ref{cor:smallfgbtd}, there exists an ${\rm
FrGBTD}(3,6^t)$ for all $t\in\{5,6,7\}$. By Proposition
\ref{prop:truncation}, there exists an ${\rm
FrGBTD}(3,(6n)^5(6g_1)\cdots (6g_s))$. Now apply Corollary
\ref{fgbtdconstruction} to obtain a special ${\rm
GBTD}_1(3,10n+1+2\sum_{i=1}^s g_i)$.
\end{proof}

\begin{lemma}
\label{lem:generalcase} A special ${\rm GBTD}_1(3,m)$ exists for odd $m\ge 7$.
\end{lemma}

 \begin{proof}
First, a special ${\rm GBTD}_1(3,m)$ can be constructed for odd $m$, $7\le m\le 95$. Details are provided in Table \ref{authority}.

 We then prove the lemma by induction on $m\ge 97$.

 Let $E=\{t: t\geq 9\}\setminus \{10,14, 15,18,
 20, 22, 26, 30, 34, 38$, $46, 60\}$.
 By Theorem \ref{thm:tdexistence},
 a TD$(7,n)$ exists for any $n\in E$.
If there exists a special GBTD$_1(3,m')$ for odd $m'$, $7\le m'\le 2n+1$,
 then apply Lemma \ref{lem:10n+1} with $3\le g_1,g_2\le n$ to obtain
 a special GBTD$_1(3,m)$ for odd $m$, $10n+7\le m\le 14n+1$.

 Hence, take $n=9$ to obtain a special GBTD$_1(3,97)$.

 Suppose there exists a GBTD$_1(3,m')$ for all odd $m'<m$. Then there exists $n\in E$ with
 $10n+7\le m\le 14n+1$. Suppose otherwise. Then there exists $n_1\in E$ such that $14n_1+1<10n_2+7$ for all $n_2>n_1$ and $n_2\in E$.
 This, together with the fact that $n_1\ge 9$, implies that $n_2-n_1> 3$ for all $n_2\in E$ and $n_2> n_1$.
 However, a quick check on $E$ gives a contradiction.

 Since $n\in E$ and there exists a special GBTD$_1(3,m')$ for all $m'\le 2n+1< 10n+7\le m$ (induction hypothesis),
 there exists a special GBTD$_1(3,m)$ and induction is complete.
 \end{proof}

 \begin{table*}
 \small
 \centering \caption{Existence of special ${\rm GBTD}_1(3,m)$}
 \label{authority}
 \begin{tabular}{p{90mm}p{60mm}}
 \hline
 Authority & $m$ \\
 \hline
 Corollary \ref{cor:smallgbtd} & 9, 11, 17, 23, 29, 35, 47, 53, 55 \\
 Lemma \ref{lem:3^rq} & 7, 13, 15, 19, 21, 25, 27, 31, 33, 37, 39, 43, 45, 49, 57, 61, 63, 67, 69, 73, 75\\
 Corollary \ref{gbtd:gt+1} with $(g,t)$ in $\{(8,5),(5,10)$, $(8,8),(7,10)\}$ & 41, 51, 65, 71\\
 Lemma \ref{lem:10n+1} with $n=5$, $g_1=4$ & 59\\
 Lemma \ref{lem:10n+1} with $n=7$, $g_1,g_2\in\{0\}\cup \{t: 3\leq t\leq 7\}$ & $\{s: 77\le s\le 95, s\mbox{ odd}\}$\\
 \hline
 \end{tabular}
 \end{table*}

 Lemma \ref{lem:generalcase}
 shows that a ${\rm GBTD}_1(3,m)$ exists for all odd $m\not=3,5$. Theorem \ref{thm:main} (vi) now follows.

 \subsection{Existence of GBTP$_1(\{2,3^*\};2m+1,m\times(2m-3))$}



 \begin{lemma}
 \label{lem:TD-igbtp}
 Suppose there exists a TD$(5,n)$. Suppose $0\leq g\leq n$ and that there exists
 an IGBTP$_1(\{2,3^*\},2m+9,(m+4)\times (2m+5);9,4\times 5)$ for
 $m\in\{n, g\}$.
 Then there exists an IGBTP$_1(\{2,3^*\},2M+9, (M+4)\times (2M+5);9,4\times 5)$, where $M=4n+g$.
 \end{lemma}

\begin{proof}
By Lemma \ref{lem:small-fgbtd-2}, there exists an ${\rm
FrGBTD}(2,2^t)$ for all $t\in\{4,5\}$. By Proposition
\ref{prop:truncation}, there exists an ${\rm
FrGBTD}(2,(2n)^4(2g))$. Now apply Proposition \ref{fgbtd+IGBTP}
to obtain an ${\rm
IGBTP}_1(\{2,3^*\},2M+9,(M+4)\times (2M+5);9,4\times 5)$.
\end{proof}

\begin{table*}
\small
\centering \caption{Existence of IGBTP$_1(\{2,3^*\},2m+9,(m+4)\times (2m+5);4\times 5)$}
 \label{GBTPauthority}
\begin{tabular}{p{90mm}p{60mm}}
 \hline
 Authority & $m$ \\
 \hline 
 Corollary \ref{cor:small-igbtp} & $\{s: 10\le s\le 57\}\setminus\{16,20,24, 28,32, 36$, $40,44,48, 50, 52,54,55,56\}$ \\
 Lemma \ref{lem:TD-igbtp} with $(n,g)\in\{(10,0), (11,0), (12,0),(13,0)$, $(11,10),(11,11),(14,0)\}$
 & 40, 44, 48, 52, 54, 55, 56 \\
 \hline
 \end{tabular}
 \end{table*}

 \begin{lemma}
 \label{lem:IGBTP}There exists an IGBTP$_1(\{2,3^*\},2m+9,(m+4)\times (2m+5);9,4\times 5)$
 for any  $m\geq 10$, except possibly for $m\in\{16,20,24,28,32,36,46,50\}$.
 \end{lemma}

 \begin{proof}
 Let $E=\{16,20,24,28,32,36,46,50\}$.
 An IGBTP$_1(\{2,3^*\},2m+9,(m+4)\times(2m+5); 9,4\times 5)$ can be constructed for $10\le m\le 57$ and $m\notin E \cup\{51\}$. 
 Details are provided in Table \ref{GBTPauthority}.
 When $m=51$, consider a TD$(5,11)$ and delete four points from a block to form a $\{4,5\}$-GDD of type $10^4 11$. 
 Proposition \ref{FC} yields an FrGBTD$(2,20^4 22)$ and hence, Proposition \ref{fgbtd+IGBTP} yields 
 an IGBTP$_1(\{2,3^*\},2m+9,(m+4)\times(2m+5); 9,4\times 5)$ with $m=51$.

We then prove the lemma by induction on $m \ge 57$.
Let $E'=\{4n+g: n\in E, 10\le g\le 13\}$ and assume the lemma is true for $n<m$.
 
When $m\notin E'$, then write $m=4n+g$ with $13\le n< m$, $n\notin E$ and $g\in\{10,11,12,13\}$.  
 Since a TD$(5,n)$ which exists by Theorem \ref{thm:tdexistence}, applying Lemma \ref{lem:TD-igbtp} with 
 the corresponding $n$ and $g$, we obtain the desired IGBTP.
 
When $m\in E'$, we have two cases.
\begin{itemize}
\item If $m=77$, the required IGBTP is given by Corollary \ref{cor:small-igbtp}.
\item Otherwise, apply Lemma \ref{lem:TD-igbtp} 
with $(n,g)$ taking values in $\{(15, 14), (15, 15),(19, 0), (18, 18)$, \\ 
$(19, 15), (23, 0), (19, 17), (22, 18), (22, 19), (27, 0), (22, 21), (25, 22), (25, 23), (31, 0), (25,25)$, \\
$(29, 22), (29, 23), (35, 0), (29, 25), (31, 30), (31, 31), (39, 0),(33, 25), (39, 38), (39, 39), (49, 0)$, \\
$(40, 37), (42, 42), (43, 39), (43, 40), (43, 41)\}$.
\end{itemize}
This completes the induction.
 \end{proof}

 \begin{lemma}
 \label{lem:GBTP} A GBTP$_1(\{2, 3^* \}, 2m + 1, m \times (2m - 3))$ exists for $m\ge 4$, except possibly for $m\in\{12,13\}$. 
 \end{lemma}

\begin{proof}
A GBTP$_1(\{2, 3^* \}; 2m + 1, m \times (2m - 3))$ can be found via computer search for $4\le m\le 11$. 
The GBTPs are listed in \cite{Cheeetal:2012online}.

For $m\in\{20,24,28,32,36,40,50,54\}$, set $M=m-5$ and we apply Proposition \ref{prop:filling hole for IGBTP}
with the GBTP$_1(\{2, 3^* \},11, 5\times 7)$ and the IGBTP$_1(\{2,3^*\},2M+11,(M+5)\times(2M+7); 11,5\times 7)$ 
constructed in Corollary \ref{cor:small-igbtp}.

Finally, for $m\ge 14$ and $m\notin\{20,24,28,32,36,40,50,54\}$, set $M=m-4$ and apply Proposition \ref{prop:filling hole for IGBTP}
with GBTP$_1(\{2, 3^* \}, 9, 4\times 5)$ and the IGBTP$_1(\{2,3^*\},2M+9,(M+4)\times(2M+5); 9,4\times 5)$ 
constructed in Lemma \ref{lem:IGBTP}.
\end{proof}

Lemma \ref{lem:GBTP}
 shows that a ${\rm GBTP}_1(\{2, 3^* \},2m + 1, m \times (2m - 3))$ exists for all $m\ge 4$, except possibly for $m\in\{12,13\}$.
 Theorem \ref{thm:main} (vii) now follows.

 \section{Conclusion}

 In this paper, we establish infinite families of ESWCs,
 whose code lengths are greater than alphabet size and
 whose relative narrowband noise error-correcting capabilities tend to a positive constant as length grows.
 The construction method used is combinatorial and reveals interesting
 interplays with equivalent combinatorial designs called
 GBTPs.
 These have enabled us to borrow ideas from combinatorial
 design theory to construct ESWCs. In return, questions on ESWCs
 offer new problems to combinatorial design theory. We expect this symbiosis
 to deepen.

\section*{Acknowledgement}

Research of Y.~M.~Chee, H.~M.~Kiah and C.~Wang is supported in part by the National Research Foundation of Singapore 
 under Research Grant NRF-CRP2-2007-03. C. Wang is also supported in part by NSFC under Grant No.10801064 and 11271280.

The authors thank Charlie Colbourn for useful discussions. 
The authors are also grateful to both anonymous reviewers for their constructive comments
and pointing out the relevant literature.

\bibliographystyle{siam}
\bibliography{mybibliography}

\begin{thebibliography}{10}

\bibitem{Abeletal:2007}
{\sc R.~J.~R. Abel, C.~J. Colbourn, and J.~H. Dinitz}, {\em Mutually orthogonal
  {L}atin squares ({MOLS})}, in The CRC Handbook of Combinatorial Designs,
  C.~J. Colbourn and J.~H. Dinitz, eds., CRC Press, Boca Raton, 2nd~ed., 2007,
  pp.~160--193.

\bibitem{Biglieri:2003}
{\sc E.~Biglieri}, {\em Coding and modulation for a horrible channel}, IEEE
  Commun. Mag., 41 (2003), pp.~92--98.

\bibitem{Bogdanovaetal:2001}
{\sc G.~T. Bogdanova, A.~E. Brouwer, S.~N. Kapralov, and P.~R.
  {\"O}sterg{\aa}rd}, {\em Error-correcting codes over an alphabet of four
  elements}, Des. Codes Cryptogr., 23 (2001), pp.~333--342.

\bibitem{Bogdanovaetal:2007}
{\sc G.~T. Bogdanova, V.~Zinoviev, and T.~J. Todorov}, {\em On the construction
  of $q$-ary equidistant codes}, Problems of Inform. Transmission, 43 (2007),
  pp.~280--302.

\bibitem{Cheeetal:2012online}
{\sc Y.~M. Chee, H.~M. Kiah, A.~C.~H. Ling, and C.~Wang}, {\em Addendum to
  optimal equitable symbol weight codes for power line communications}, {\tt
  \url{http://www1.spms.ntu.edu.sg/~kiah0001/optimalESWC.html}}.

\bibitem{Cheeetal:2012b}
{\sc Y.~M. Chee, H.~M. Kiah, A.~C.~H. Ling, and C.~M. Wang}, {\em Optimal
  equitable symbol weight codes for power line communications}, in Proc. IEEE
  Intl. Symp. Inform. Theory, 2012, pp.~671--675.

\bibitem{Cheeetal:2013tc}
{\sc Y.~M. Chee, H.~M. Kiah, P.~Purkayastha, and C.~Wang}, {\em Importance of
  symbol equity in coded modulation for power line communications}, IEEE Trans.
  Commun., 61 (2013), pp.~4381--4390.

\bibitem{Cheeetal:2013b}
{\sc Y.~M. Chee, H.~M. Kiah, and C.~Wang}, {\em Generalized balanced tournament
  designs with block size four}, Electron. J. Combin., 20 (2013), pp.~P51, 14
  pp. (electronic).

\bibitem{Colbournetal:2004}
{\sc C.~J. Colbourn, T.~Kl{\o}ve, and A.~C.~H. Ling}, {\em Permutation arrays
  for powerline communication and mutually orthogonal {L}atin squares}, IEEE
  Trans. Inform. Theory, 50 (2004), pp.~1289--1291.

\bibitem{Daietal:2013}
{\sc P.~Dai, J.~Wang, and J.~Yin}, {\em Two series of equitable symbol weight
  codes meeting the {P}lotkin bound}, Des. Codes Cryptogr.,  (2013).

\bibitem{DingYin:2005b}
{\sc C.~Ding and J.~Yin}, {\em Combinatorial constructions of optimal
  constant-composition codes}, IEEE Trans. Inform. Theory, 51 (2005),
  pp.~3671--3674.

\bibitem{DingYin:2006}
\leavevmode\vrule height 2pt depth -1.6pt width 23pt, {\em A construction of
  optimal constant composition codes}, Des. Codes Cryptogr., 40 (2006),
  pp.~157--165.

\bibitem{Dukes:2012}
{\sc P.~J. Dukes}, {\em Coding with injections}, Des. Codes Cryptogr., 65
  (2012), pp.~213--222.

\bibitem{Dzungetal:2011}
{\sc D.~Dzung, I.~Berganza, and A.~Sendin}, {\em Evolution of powerline
  communications for smart distribution: from ripple control to {OFDM}}, in
  Proc. IEEE Intl. Symp. Power Line Commun. and its Applicat., 2011,
  pp.~474--478.

\bibitem{FranklDeza:1977}
{\sc P.~Frankl and M.~Deza}, {\em On the maximum number of permutations with
  given maximal or minimal distance}, J. Comb. Theory Ser. A, 22 (1977),
  pp.~352--360.

\bibitem{Haidineetal:2011}
{\sc A.~Haidine, B.~Adebisi, A.~Treytl, H.~Pille, B.~Honary, and A.~Portnoy},
  {\em High-speed narrowband {PLC} in smart grid landscape --
  state-of-the-art}, in Proc. IEEE Intl. Symp. Power Line Commun. and its
  Applicat., 2011, pp.~468--473.

\bibitem{Huczynska:2010}
{\sc S.~Huczynska}, {\em Equidistant frequency permutation arrays and related
  constant composition codes}, Des. Codes Cryptogr., 54 (2010), pp.~109--120.

\bibitem{LamkenVanstone:1989}
{\sc E.~Lamken and S.~Vanstone}, {\em Balanced tournament designs and related
  topics}, Discrete Math., 77 (1989), pp.~159--176.

\bibitem{Lamken:1990}
{\sc E.~R. Lamken}, {\em Generalized balanced tournament designs}, Trans. Amer.
  Math. Soc., 318 (1990), pp.~473--490.

\bibitem{Lamken:1992}
\leavevmode\vrule height 2pt depth -1.6pt width 23pt, {\em Existence results
  for generalized balanced tournament designs with block size 3}, Des. Codes
  Cryptogr., 3 (1992), pp.~33--61.

\bibitem{Lamken:1994}
\leavevmode\vrule height 2pt depth -1.6pt width 23pt, {\em Constructions for
  generalized balanced tournament designs}, Discrete Math., 131 (1994),
  pp.~127--151.

\bibitem{Lamken:1995}
{\sc E.~R. Lamken}, {\em The existence of doubly resolvable $(v, 3, 2)$-bibds},
  J. Comb. Theory Ser. A, 72 (1995), pp.~50--76.

\bibitem{Lamken:1997}
\leavevmode\vrule height 2pt depth -1.6pt width 23pt, {\em The existence of
  partitioned generalized balanced tournament designs with block size 3}, Des.
  Codes Cryptogr., 11 (1997), pp.~37--71.

\bibitem{Liuetal:2011}
{\sc J.~Liu, B.~Zhao, L.~Geng, Z.~Yuan, and Y.~Wang}, {\em Communication
  performance of broadband {PLC} technologies for smart grid}, in Proc. IEEE
  Intl. Symp. Power Line Commun. and its Applicat., 2011, pp.~491--496.

\bibitem{Pavlidouetal:2003}
{\sc N.~Pavlidou, A.~J.~H. Vinck, J.~Yazdani, and B.~Honary}, {\em Power line
  communications: state of the art and future trends}, IEEE Commun. Mag., 41
  (2003), pp.~34--40.

\bibitem{Schellenbergetal:1977}
{\sc P.~Schellenberg, G.~Van~Rees, and S.~Vanstone}, {\em The existence of
  balanced tournament designs}, Ars Combinatoria, 3 (1977), pp.~303--318.

\bibitem{Schwartz:2009}
{\sc M.~Schwartz}, {\em Carrier-wave telephony over power lines: Early history
  [history of communications]}, IEEE Commun. Mag., 47 (2009), pp.~14--18.

\bibitem{SemakovZinoviev:1968}
{\sc N.~V. Semakov and V.~A. Zinoviev}, {\em Equidistant $q$-ary codes with
  maximal distance and resolvable balanced incomplete block designs}, Problemy
  Pereda\v{c}i Informacii, 4 (1968), pp.~3--10.

\bibitem{Vinck:2000}
{\sc A.~J.~H. Vinck}, {\em Coded modulation for power line communications},
  AE\"{U} - Int J. Electron. Commun., 54 (2000), pp.~45--49.

\bibitem{Wilson:1972d}
{\sc R.~M. Wilson}, {\em Cyclotomy and difference families in elementary
  abelian groups}, J. Number Theory, 4 (1972), pp.~17--47.

\bibitem{Wilson:1972a}
{\sc R.~M. Wilson}, {\em An existence theory for pairwise balanced designs.
  {I}. {C}omposition theorems and morphisms}, J. Combin. Theory Ser. A, 13
  (1972), pp.~220--245.

\bibitem{Yinetal:2008}
{\sc J.~Yin, J.~Yan, and C.~Wang}, {\em Generalized balanced tournament designs
  and related codes}, Des. Codes Cryptogr., 46 (2008), pp.~211--230.

\bibitem{ZhangYang:2011}
{\sc W.~Zhang and L.~Yang}, {\em {SC-FDMA} for uplink smart meter transmission
  over low voltage power lines}, in Proc. IEEE Intl. Symp. Power Line Commun.
  and its Applicat., 2011, pp.~497--502.

\end{thebibliography}







\end{document}